\documentclass[12pt]{article}

\usepackage{amsmath}
\usepackage{amssymb}
\usepackage{amsthm}
\usepackage{amsfonts}
\usepackage{latexsym}
\usepackage{cite}
\usepackage{psfrag}
\usepackage{epsfig}
\usepackage{graphicx}
\usepackage{mathrsfs}
\usepackage{url}
\usepackage{color}
\usepackage{cancel}
\usepackage{eufrak}
\usepackage{multirow}
\usepackage{hyperref}

\makeatletter
\@addtoreset{equation}{section}
\makeatother

\newtheorem{theorem}{Theorem} 
\newtheorem{lemma}{Lemma}
\newtheorem{corollary}{Corollary}
\newtheorem{conjecture}{Conjecture}

\theoremstyle{definition}
\newtheorem{definition}{Definition}

\newtheorem{remark}{Remark}

\newcommand{\I}{\mathcal{I}}
\newcommand{\LL}{\mathcal{L}}
\newcommand{\OO}{\mathcal{O}}
\newcommand{\TT}{\mathcal{T}}
\newcommand{\W}{\mathcal{W}}

\newcommand{\C}{\mathscr{C}}
\newcommand{\M}{\mathscr{M}}
\newcommand{\N}{\mathscr{N}}

\newcommand{\PG}{\mathrm{PG}}
\newcommand{\RC}{\mathrm{RC}}
\newcommand{\RA}{\mathrm{RA}}
\newcommand{\Tr}{\mathrm{T}}
\newcommand{\IC}{\mathrm{IC}}
\newcommand{\IA}{\mathrm{IA}}
\newcommand{\UG}{\mathrm{U\Gamma}}
\newcommand{\UnG}{\mathrm{Un\Gamma}}
\newcommand{\EG}{\mathrm{E\Gamma}}
\newcommand{\EnG}{\mathrm{En\Gamma}}
\newcommand{\Ar}{\mathrm{A}}
\newcommand{\EA}{\mathrm{EA}}
\newcommand{\TO}{\mathrm{TO}}

\newcommand{\MM}{\mathbf{M}}
\newcommand{\Pf}{\mathbf{P}}

\newcommand{\F}{\mathbb{F}}
\newcommand{\Pb}{\mathbb{P}}
\newcommand{\Lb}{\mathbb{L}}

\newcommand{\A}{\mathfrak{A}}
\newcommand{\Lk}{\mathfrak{L}}
\newcommand{\Mk}{\mathfrak{M}}
\newcommand{\Pk}{\mathfrak{P}}
\newcommand{\pk}{\mathfrak{p}}

\newcommand{\od}{\mathrm{od}}
\newcommand{\ev}{\mathrm{ev}}

\newcommand{\T}{\text}
\newcommand{\db}{\displaybreak[3]}
\newcommand{\dbn}{\displaybreak[3]\notag}
\newcommand{\nt}{\notag}

\begin{document}
\title{Twisted cubic and point-line incidence matrix in $\PG(3,q)$
\date{}
}
\maketitle

\begin{center}
{\sc Alexander A. Davydov
\footnote{A.A. Davydov ORCID \url{https://orcid.org/0000-0002-5827-4560}}
}\\
{\sc\small Institute for Information Transmission Problems (Kharkevich institute)}\\
 {\sc\small Russian Academy of Sciences}\\
 {\sc\small Moscow, 127051, Russian Federation}\\\emph {E-mail address:} adav@iitp.ru\medskip\\
 {\sc Stefano Marcugini
 \footnote{S. Marcugini ORCID \url{https://orcid.org/0000-0002-7961-0260}},
 Fernanda Pambianco
 \footnote{F. Pambianco ORCID \url{https://orcid.org/0000-0001-5476-5365}}
 }\\
 {\sc\small Department of  Mathematics  and Computer Science,  Perugia University,}\\
 {\sc\small Perugia, 06123, Italy}\\
 \emph{E-mail address:} \{stefano.marcugini, fernanda.pambianco\}@unipg.it
\end{center}

\begin{abstract}
We consider the structure of the point-line incidence matrix of the projective space $\PG(3,q)$ connected with orbits of points and lines under the stabilizer group of the twisted cubic.  Structures of submatrices with incidences between a union of line orbits and an orbit of points are investigated.
 For the unions consisting of two or three line orbits, the original submatrices are split into new ones, in which the incidences  are also considered.
For each submatrix (apart from the ones corresponding to a special type of lines), the numbers of lines through every point and of points lying on  every line are obtained. This corresponds to the numbers of ones in columns and rows of the submatrices.
 \end{abstract}

\textbf{Keywords:} Twisted cubic, Projective space, Incidence matrix, Orbits of lines

\textbf{Mathematics Subject Classification (2010).} 51E21, 51E22

\section{Introduction}
Let $\F_{q}$ be the Galois field with $q$ elements.
Let $\PG(N,q)$ be the $N$-dimensio\-nal projective space over $\F_q$. An $n$-arc in  $\PG(N,q)$, with $n\ge N + 1\ge3$, is a set of $n$ points such that no $N +1$ points belong to the same hyperplane of $\PG(N,q)$, see \cite{BallLavrauw} and the references therein.   For an introduction to projective geometry over finite fields see \cite{Hirs_PGFF,HirsStor-2001,HirsThas-2015}.

In $\PG(N,q)$, $2\le N\le q-2$, a normal rational curve is a $(q+1)$-arc projectively equivalent to the arc
$\{(t^N,t^{N-1},\ldots,t^2,t,1):t\in \F_q\}\cup \{(1,0,\ldots,0)\}$. In $\PG(3,q)$, the normal rational curve is called a  \emph{twisted cubic} \cite{Hirs_PG3q,HirsThas-2015}.

The twisted cubic has many interesting properties and is connected with distinct combinatorial and applied problems, see e.g. \cite{BDMP-TwCub,BlokPelSzo,BonPolvTwCub,BrHirsTwCub,CLPolvT_Spr,CasseGlynn82,CasseGlynn84,CosHirsStTwCub,DMP_RSCoset,DMP_PlLineInc,DMP_OrbLine,%
GiulVincTwCub,GulLav,Hirs_PG3q,HirsStor-2001,HirsThas-2015,LunarPolv,ZanZuan2010} and the references therein. In particular, using properties of the twisted cubic, spreads in $\PG(3,q)$ are studied \cite{BrHirsTwCub,CLPolvT_Spr,LunarPolv},  optimal multiple covering codes are constructed \cite{BDMP-TwCub},  the weight distributions of cosets and their leaders for the Reed-Solomon codes are obtained \cite{BlokPelSzo,DMP_RSCoset}, the three-level secret sharing schemes are considered \cite{GiulVincTwCub}.

In investigations of the twisted cubic, an important direction is to determine the matrices of the incidences between points, planes, and lines partitioned into orbits under the group $G_q$ fixing the cubic. The orbits of planes and points are known and described in detail \cite{Hirs_PG3q}. The \emph{point-plane} incidence matrix of $\PG(3,q)$ for all $q\ge2$ is given in \cite{BDMP-TwCub} where the numbers of distinct planes through distinct points and, conversely, the numbers of distinct points lying in distinct planes are obtained. (By ``distinct planes" we mean ``planes from distinct orbits", and similarly for points and lines.)

For plane-line and point-line incidence matrices a description of line orbits is needed. In \cite{Hirs_PG3q}, the lines in $\PG(3,q)$ are partitioned into classes, each of which is a union of line orbits under $G_q$; see Section 2.2. Apart from one class (which is denoted by $\OO_6$), the number and the structure of the orbits forming those unions are independently considered by distinct methods in \cite[Sections 3, 8]{DMP_OrbLine} (for all $q\ge2$),  \cite[Section 7]{BlokPelSzo} (for all $q\ge23$), and \cite{GulLav}  (for finite fields of characteristic $> 3$); see also the references therein.

The classification of the line orbits in the class $\OO_6$ is an open problem.

The results on line orbits from \cite{BlokPelSzo,DMP_OrbLine,GulLav} are in accordance with each other. The representation and description of the orbits in these papers are distinct; in particular, in \cite{DMP_OrbLine}, the orbits are given in a form which is convenient for the investigations in \cite{DMP_PlLineInc}.
More precisely, using the representation of the line orbits in \cite{DMP_OrbLine}, the \emph{plane-line} incidence matrix of $\PG(3,q)$ is given in \cite{DMP_PlLineInc}
where, apart from $\OO_6$, for all $q\ge2$, the numbers of distinct planes through distinct lines and, vice versa, the numbers of distinct lines lying in distinct planes are obtained. For $\OO_6$, the corresponding average values are calculated.

In \cite{GulLav}, apart from $\OO_6$, for odd $q\not\equiv0\pmod3$ the numbers of distinct planes through distinct lines (called ``the plane orbit distribution of a line") and the numbers of distinct points lying on distinct lines
(called ``the point orbit distribution of a line") are obtained.
For finite fields of characteristic $> 3$, the results of \cite{GulLav} on ``the plane orbit distribution of a line" are in accordance with those from \cite{DMP_PlLineInc} on plane-line incidence matrix.

The results of \cite{GulLav} on ``the point orbit distribution of a line" are an important step towards the point-line incidence matrix. However, these results are obtained only for odd $q\not\equiv0\pmod3$ and
the computation of  the numbers of distinct lines through distinct points has been left open.

In this paper, we obtain the \emph{point-plane} incidence matrix for all $q\ge2$, leaving open the questions related to $\OO_6$.

We consider the structure of the point-line incidence matrix with respect to $G_q$.
We use the partitions of planes and lines into orbits and unions of orbits under the group $G_q$, as described in \cite{DMP_OrbLine,Hirs_PG3q}. We search the structures of the submatrices with incidences between an orbit of points and a union of line orbits.
 For the unions consisting of two or three line orbits, the original submatrices are split into new ones, in which the incidences  are also considered.
For each submatrix (apart from the ones related to $\OO_6$), the numbers of distinct points lying on distinct lines and, conversely,   the numbers of distinct lines through distinct points are obtained.
This corresponds to the numbers of ones in columns and rows of the submatrices.

The results noted are obtained for all $q\ge2$ including even $q$ and $q\equiv0\pmod3$. Thus, the gaps of \cite {GulLav} in the point-line incidence matrix are filled.

For $\OO_6$, some average and cumulative values are calculated.

Many submatrices considered are configurations in the sense of \cite
{GroppConfig}, see Definition~\ref{def2_config} in Section \ref{subsec_incid}. Such configurations are useful in several distinct areas, in particular, to construct bipartite graph codes without the so-called 4-cycles, see e.g.  \cite{BargZem,DGMP_BipGraph,HohJust} and the references therein.

The paper is organized as follows. Section \ref{sec_prelimin} contains preliminaries. In Section~\ref{sec_mainres}, the main results of this paper are summarized. Some useful relations are given in Section~\ref{sec:useful}.  The numbers of distinct points lying on distinct lines and, vice versa,   the numbers of distinct lines through distinct points are obtained in Sections \ref{sec:results_q_ne0} (for even and odd $q\not\equiv0\pmod3$) and \ref{sec:results_q=0}
(for $q\equiv0\pmod3$). Some  general results are given in Section \ref{sec:gen res}.

\section{Preliminaries}\label{sec_prelimin}
Throughout the paper, we consider orbits of lines and points under $G_q$ apart from Theorem \ref{th3:q=2 3 4} in Section \ref{sec_mainres}.
\subsection{Twisted cubic}\label{subset_twis_cub}
In this subsection, including Theorem \ref{th2_Hirs}, we summarize  some results from \cite{Hirs_PG3q} useful in this paper.

The space $\PG(N,q)$ contains $\theta_{N,q}$ points and hyperplanes, and $\beta_{N,q}$ lines;
\begin{align}\label{eq1_theta_lambda}
 \theta_{N,q}=\frac{q^{N+1}-1}{q-1} ,~\beta_{N,q}=\frac{(q^{N+1}-1)(q^{N+1}-q)}{(q^2-1)(q^2-q)}\,.
\end{align}
Let $\boldsymbol{\pi}(c_0,c_1,c_2,c_3)$ be the plane of $\PG(3,q)$ with equation
\begin{align}\label{eq2_plane}
  c_0x_0+c_1x_1+c_2x_2+c_3x_3=0,~c_i\in\F_q.
\end{align}
We denote $\F_{q}^*=\F_{q}\setminus\{0\}$, $\F_q^+=\F_q\cup\{\infty\}$. Let $\Pf(x_0,x_1,x_2,x_3)\in\PG(3,q)$ be a point  with homogeneous coordinates $x_i\in\F_{q}$.
Let  $P(t)$ be a point with
\begin{align}\label{eq2:P(t)}
t\in\F_q^+;~  P(t)=\Pf(t^3,t^2,t,1)\text{ if }t\in\F_q;~~P(\infty)=\Pf(1,0,0,0).
\end{align}
Let $\C\subset\PG(3,q)$ be the \emph{twisted cubic} in the canonical form
\begin{align}\label{eq2_cubic}
&\C=\{P_1,P_2,\ldots,P_{q+1}\}=\{P(t)\,|\,t\in\F_q^+\}.
\end{align}
where $P_1,\ldots,P_{q+1}$  are points no four of which are coplanar.

The \emph{osculating plane} $\pi_\T{osc}(t)$ at the  point $P(t)\in\C$ has the form
\begin{align}\label{eq2_osc_plane}
&\pi_\T{osc}(t)=\boldsymbol{\pi}(1,-3t,3t^2,-t^3)\T{ if }t\in\F_q; ~\pi_\T{osc}(\infty)=\boldsymbol{\pi}(0,0,0,1).
\end{align}
 The $q+1$ osculating planes form the \emph{osculating developable} $\mathrm{\Gamma}$ to $\C$.
For $q\equiv0\pmod3$, the osculating developable is a \emph{pencil of planes}.

\begin{definition}
 \begin{description}
   \item[(i)]
A \emph{chord} of $\C$ is a line through a pair of real points of $\C$ or a pair of complex conjugate points. If the real points coincide with each other, the chord is a \emph{tangent} to $\C$; if they are distinct, we have a \emph{real chord}. For a pair of complex conjugate points, we have an \emph{imaginary chord}.

   \item[(ii)] An \emph{axis} of $\mathrm{\Gamma}$ is a line of $\PG(3,q)$ which is the intersection of a pair of real planes or complex conjugate planes of $\mathrm{\Gamma}$. If the real planes coincide with each other, the axis is a \emph{tangent} to $\C$; if they are distinct it is a \emph{real axis}. For complex conjugate planes, we have an \emph{imaginary axis}.
 \end{description}
\end{definition}
The null polarity $\A$ \cite[Sections 2.1.5, 5.3]{Hirs_PGFF}, \cite[Theorem 21.1.2]{Hirs_PG3q} is given by
\begin{align}\label{eq2_null_pol}
&\Pf(x_0,x_1,x_2,x_3)\A=\boldsymbol{\pi}(x_3,-3x_2,3x_1,-x_0),~q\not\equiv0\pmod3.
\end{align}

\textbf{Notation 1}
~We consider $q\equiv\xi\pmod3$, $\xi\in\{-1,0,1\}$. Values depending of $\xi$ are
 noted by remarks or by superscripts ``$(\xi)$''. The remarks and superscripts ``$(\xi)$'' are not used
 if a value is the same for all $q$ or a property holds for all $q$, or it is not relevant, or it is clear by the context. If a value is the same for $\xi=-1,1$, then one may use the superscript ``$\ne0$''. Also, in superscripts, instead of ``$\bullet$'', one can write ``$\mathrm{ev}$'' for even $q$ or ``$\mathrm{od}$'' for odd $q$. If a value is the same for even and odd $q$, we may omit ``$\bullet$''.

The following notation is used.
\begin{align*}
  &G_q && \T{the group of projectivities in } \PG(3,q) \T{ fixing }\C;\db  \\
  &\mathbf{Z}_n&&\T{cyclic group of order }n;\db  \\
  &\mathbf{S}_n&&\T{symmetric group of degree }n;\db  \\
&A^{tr}&&\T{the transposed matrix of }A;\db \\
&\#S&&\T{the cardinality of a set }S;\db\\
&\overline{AB}&&\T{the line through the points $A$ and }B;\db\\
&\triangleq&&\T{the sign ``equality by definition"}.\db\\
&&&\T{\textbf{Types $\pi$ of planes:}}\db\\
&\mathrm{\Gamma}\T{-plane}  &&\T{an osculating plane of }\mathrm{\Gamma};\db \\
&d_\C\T{-plane}&&\T{a plane containing \emph{exactly} $d$ distinct points of }\C,~d=0,2,3;\db \\
&\overline{1_\C}\T{-plane}&&\T{a plane not in $\mathrm{\Gamma}$ containing \emph{exactly} 1 point of }\C;\db \\
&\Pk&&\T{the list of possible types $\pi$ of planes},~\Pk\triangleq\{\mathrm{\Gamma},2_\C,3_\C,\overline{1_\C},0_\C\};\db\\
&\pi\T{-plane}&&\T{a plane of type }\pi\in\Pk; \db\\
&\N_\pi&&\T{the orbit of $\pi$-planes under }G_q,~\pi\in\Pk.\db\\
&&&\T{\textbf{Types $\pk$ of points with respect to the twisted cubic $\C$:}}\db\\
&\C\T{-point}&&\T{a point  of }\C;\db\\
&\mu_\mathrm{\Gamma}\T{-point}&&\T{a point  off $\C$ lying on \emph{exactly} $\mu$ distinct osculating planes;}\db\\
&\Tr\T{-point}&&\T{a point  off $\C$  on a tangent to $\C$ for }\xi\ne0;\db\\
&\TO\T{-point}&&\T{a point  off $\C$ on a tangent and one osculating plane for }\xi=0;\db\\
&\RC\T{-point}&&\T{a point  off $\C$  on a real chord;}\db\\
&\IC\T{-point}&&\T{a point  on an imaginary chord (it always is off $\C$);}\\
&\Mk^{(\xi)}&&\T{the list of possible types $\pk$ of points},\db\\
&&&\Mk^{(\ne0)}\triangleq\{\C,0_\mathrm{\Gamma},1_\mathrm{\Gamma},3_\mathrm{\Gamma},\Tr,\RC,\IC\},\db\\
&&&\Mk^{(0)}\triangleq\{\C,(q+1)_\mathrm{\Gamma},\TO,\RC,\IC\};\db\\
&\M_\pk^{(\xi)}&&\T{the orbit of $\pk$-points under }G_q,~\pk\in\Mk^{(\xi)}.\db\\
&&&\T{\textbf{Types $\lambda$ of lines with respect to the twisted cubic $\C$:}}\db\\
&\RC\T{-line}&&\T{a real chord  of $\C$;}\db \\
&\RA\T{-line}&&\T{a real axis of $\mathrm{\Gamma}$ for }\xi\ne0;\db \\
&\Tr\T{-line}&&\T{a tangent to $\C$};\db \\
&\IC\T{-line}&&\T{an imaginary chord  of $\C$;}\db \\
&\IA\T{-line}&&\T{an imaginary axis of $\mathrm{\Gamma}$ for }\xi\ne0;\db \\
&\UG&&\T{a non-tangent unisecant in a $\mathrm{\Gamma}$-plane;}\db \\
&\UnG\T{-line}&&\T{a unisecant not in a $\mathrm{\Gamma}$-plane (it is always non-tangent);}\db \\
&\EG\T{-line}&&\T{an external line in a $\mathrm{\Gamma}$-plane (it cannot be a chord);}\db \\
&\EnG\T{-line}&&\T{an external line, other than a chord, not in a $\mathrm{\Gamma}$-plane;}\db \\
&\Ar\T{-line}&&\T{the axis of the pencil of $\mathrm{\Gamma}$-planes for }\xi=0;\db\\
&\EA\T{-line}&&\T{an external line meeting the axis of $\mathrm{\Gamma}$ for }\xi=0;\db\\
&\Lk^{(\xi)}&&\T{the list of possible types $\lambda$ of lines},\db\\
&&&\Lk^{(\ne0)}\triangleq\{\RC,\RA,\Tr,\IC,\IA,\UG,\UnG,\EG,\EnG\}\T{ for }\xi\ne0,\db\\
&&&\Lk^{(0)}\triangleq\{\RC,\Tr,\IC,\UG,\UnG, \EnG,\Ar,\EA\}\T{ for }\xi=0;\db\\
&\lambda\T{-line}&&\T{a line of type }\lambda\in\Lk^{(\xi)};\db\\
&&&\textbf{Orbits of lines.\,Plane-line incidence matrix.}\, \pi\in\Pk,\lambda\in\Lk^{(\xi)}\db\\
&L_{\lambda\mathrm{\Sigma}}^{(\xi)\bullet}&&\T{the total number of orbits of $\lambda$-lines};\db\\
&\OO_\lambda&&\T{the union (class) of all $L_{\lambda\mathrm{\Sigma}}^{(\xi)\bullet}$ orbits of $\lambda$-lines};\\
&\OO_{\lambda_j}&&\T{the $j$-th orbit of the class }\OO_\lambda,~ j=1,\ldots,L_{\lambda\mathrm{\Sigma}}^{(\xi)\bullet},~\OO_\lambda=\bigcup_{j=1}^{L_{\lambda\mathrm{\Sigma}}^{(\xi)\bullet}}\OO_{\lambda_j};\db\\
&\OO_{i_j}&&\T{the $j$-th orbit of the class }\OO_i;\db\\
&\lambda_j\T{-lines}&&\lambda\T{-lines forming the $j$-th orbit $\OO_{\lambda_j}$ of the class }O_{\lambda};\db\\
&\mathrm{\Lambda}_{\lambda_j,\pi}^{(\xi)\bullet}&&\T{the number of lines from an orbit $\OO_{\lambda_j}$ in a $\pi$-plane};\db\\
&\mathrm{\Lambda}_{\lambda,\pi}^{(\xi)\bullet}&&\T{the total number of $\lambda$-lines in a $\pi$-plane};\db\\
&\mathrm{\Pi}_{\pi,\lambda_j}^{(\xi)\bullet} &&\T{the exact number of $\pi$-planes through a line of an orbit $\OO_{\lambda_j}$};\db\\
&\mathrm{\Pi}_{\pi,\lambda}^{(\xi)\bullet}&&\T{the average number of $\pi$-planes through a $\lambda$-line over all the}\db\\
&&&\T{$\lambda$-lines; if the class $\OO_\lambda$ consists of \emph{a single orbit} then }\mathrm{\Pi}_{\pi,\lambda}^{(\xi)}\db\\
&&&\T{is \emph{the exact number} of $\pi$-planes through each $\lambda$-line};\db\\
&\I^{\mathrm{\Pi}\mathrm{\Lambda}}&&\T{the $\beta_{3,q}\times\theta_{3,q}$ plane-line incidence matrix  of }\PG(3,q);\db\\
&\I_{\pi,\lambda}^{\mathrm{\Pi}\mathrm{\Lambda}}&&\T{the $\#\OO_\lambda\times\#\N_\pi$ submatrix of  $\I^{\mathrm{\Pi}\mathrm{\Lambda}}$ with incidences between }\db\\
&&&\T{$\pi$-planes and $\lambda$-lines}; \db\\
&\I_{\pi,\lambda_j}^{\mathrm{\Pi}\mathrm{\Lambda}}&&\T{the $\#\OO_{\lambda_j}\times\#\N_\pi$ submatrix of $\I^{\mathrm{\Pi}\mathrm{\Lambda}}_{\pi,\lambda}$ with incidences between }\db\\
&&&\T{$\pi$-planes and $\lambda_j$-lines}.
\end{align*}

\begin{theorem}\label{th2_Hirs}
\emph{\cite[Chapter 21]{Hirs_PG3q}} The following properties of the twisted cubic $\C$ of \eqref{eq2_cubic} hold:
  \begin{align}
  &\T{(i)}\T{ The group $G_q$ acts triply transitively on }\C.\dbn\\
  & \T{Also, }G_q\cong PGL(2,q)~\T{for }q\ge5;\dbn \\
 &\phantom{\T{Also, }} G_4\cong\mathbf{S}_5\cong P\mathrm{\Gamma} L(2,4)\cong\mathbf{Z}_2PGL(2,4);~ G_3\cong\mathbf{S}_4\mathbf{Z}_2^3;~
  G_2\cong\mathbf{S}_3\mathbf{Z}_2^3.\dbn\\
  &\T{ The matrix $\MM$ corresponding to a projectivity of $G_q$ has the general form}\dbn\\
& \label{eq2_M} \mathbf{M}=\left[
 \begin{array}{cccc}
 a^3&a^2c&ac^2&c^3\\
 3a^2b&a^2d+2abc&bc^2+2acd&3c^2d\\
 3ab^2&b^2c+2abd&ad^2+2bcd&3cd^2\\
 b^3&b^2d&bd^2&d^3
 \end{array}
  \right],~a,b,c,d\in\F_q,~ ad-bc\ne0.
\end{align}

(ii) (a) Under $G_q$, $q\ge5$, there are the following five orbits $\N_j$ of planes:
\begin{align}\label{eq2_plane orbit_gen}
   &\N_1=\N_\mathrm{\Gamma}=\{\mathrm{\Gamma}\T{-planes}\},~~~\#\N_1=\#\N_\mathrm{\Gamma}=q+1;\db\\
   &\N_{2}=\N_{2_\C}=\{2_\C\T{-planes}\},~\#\N_2=\#\N_{2_\C}=q^2+q;\dbn\\
   &\N_{3}=\N_{3_\C}=\{3_\C\T{-planes}\},~\#\N_3=\#\N_{3_\C}=(q^3-q)/6;\dbn\\
 &\N_{4}=\N_{\overline{1_\C}}=\{\overline{1_\C}\T{-planes}\},~\#\N_4=\#\N_{\overline{1_\C}}=(q^3-q)/2;\dbn\\
   & \N_{5}=\N_{0_\C}=\{0_\C\T{-planes}\},~\#\N_5=\#\N_{0_\C}=(q^3-q)/3.\nt
 \end{align}

(b) For $q\not\equiv0\pmod 3$, the five orbits $\M_j^{(\ne0)}$ of points are as follows:
  \begin{align}\label{eq2_point_orbits_gen}
&\M_1^{(\ne0)}=\M_\C^{(\ne0)}=\{\C\T{-points}\},~\M_2^{(\ne0)}=\M_\Tr^{(\ne0)}=\{\Tr\T{-points}\},\db\\
&\M_3^{(\ne0)}=\M_{3_\mathrm{\Gamma}}^{(\ne0)}=\{3_\mathrm{\Gamma}\T{-points}\},~\M_4^{(\ne0)}=\M_{1_\mathrm{\Gamma}}^{(\ne0)}=\{1_\mathrm{\Gamma}\T{-points}\},\dbn\\
&~\M_5^{(\ne0)}=\M_{0_\mathrm{\Gamma}}^{(\ne0)}=\{0_\mathrm{\Gamma}\T{-points}\};~\#\M_j^{(\ne0)}=\#\N_j,~j=1,\ldots,5.\dbn\\
\label{eq2_=1_orbit_point}
 &\T{For } q\equiv1\pmod 3,~ \M_{3_\mathrm{\Gamma}}^{(1)}\cup\M_{0_\mathrm{\Gamma}}^{(1)}=\{\RC\T{-points}\}, ~ \M_{1_\mathrm{\Gamma}}^{(1)}=\{\IC\T{-points}\};\db\\
 &\T{for } q\equiv-1\pmod 3,~\M_{3_\mathrm{\Gamma}}^{(-1)}\cup\M_{0_\mathrm{\Gamma}}^{(-1)}=\{\IC\T{-points}\},~
 \M_{1_\mathrm{\Gamma}}^{(-1)}=\{\RC\T{-points}\}.\nt
\end{align}

(c) For $q\equiv0\pmod 3$, the five orbits $\M_j^{(0)}$ of points are as follows:
\begin{align}\label{eq2_=0_orbit_point}
&\M_1^{(0)}=\M_\C^{(0)}=\{\C\T{-points}\},~\M_2^{(0)}=\M_{(q+1)_\mathrm{\Gamma}}^{(0)}=\{(q+1)_\mathrm{\Gamma}\T{-points}\},\db\\
&~\#\M_1^{(0)}=\#\M_\C^{(0)}=\#\M_2^{(0)}=\#\M_{(q+1)_\mathrm{\Gamma}}^{(0)}=q+1;\dbn\\
&~\M_3^{(0)}=\M_\TO^{(0)}=\{\TO\T{-points}\},~\#\M_3^{(0)}=\#\M_\TO^{(0)}=q^2-1;\dbn\\
&\M_4^{(0)}=\M_\RC^{(0)}=\{\RC\T{-points}\},~M_5^{(0)}=\M_\IC^{(0)}=\{\IC\T{-points}\},\dbn\\
&\#\M_4^{(0)}=\#\M_\RC^{(0)}=\#\M_5^{(0)}=\#\M_\IC^{(0)}=(q^3-q)/2.\nt
\end{align}

(iii) Let $q\not\equiv0\pmod3$. The null polarity $\A$ \eqref{eq2_null_pol} interchanges $\C$ and $\mathrm{\Gamma}$ and their corresponding chords and axes. We have
 \begin{align}\label{eq2:MiU=Ni}
&  \M_j^{(\ne0)}\A=\N_j,~j=1,\ldots,5; ~\M_\C^{(\ne0)}\A=\N_\mathrm{\Gamma},~\M_\Tr^{(\ne0)}\A=\N_{2_\C},\db\\
&\M_{3_\mathrm{\Gamma}}^{(\ne0)}\A=\N_{3_\C},~\M_{1_\mathrm{\Gamma}}^{(\ne0)}\A=\N_{\overline{1_\C}},~
\M_{0_\mathrm{\Gamma}}^{(\ne0)}\A=\N_{0_\C}.\nt
 \end{align}

 (iv) For all $q$, no two chords of $\C$ meet off $\C$.
 Every point off $\C$ lies on exactly one chord of $\C$.

 (v) Let $q\not\equiv0\pmod3$. No two axes of $\mathrm{\Gamma}$ meet unless they lie in the same plane of $\mathrm{\Gamma}$.
  Every plane not in $\mathrm{\Gamma}$ contains exactly one axis of $\mathrm{\Gamma}$.
\end{theorem}

\subsection{Orbits of lines under the stabilizer group $G_q$ of the twisted cubic}
\begin{theorem} \label{th2:MAGMA}
\emph{\cite[Section 8]{DMP_OrbLine}}
Let $q\equiv\xi\pmod3$, $\xi\in\{1,-1,0\}$.
\begin{description}
  \item[(i)] Let $5\le q\le 37$ and $q=64$. Then

$\mathbf{(a)}$  For the total number $L_{\EnG\mathrm{\Sigma}}^{(\xi)\bullet}$ of orbits of $\EnG$-lines we have
\begin{align}\label{eq2:L_EnG}
&L_{\EnG\mathrm{\Sigma}}^{(\xi)\od}=2q-3+\xi,~L_{\EnG\mathrm{\Sigma}}^{(\xi)\ev}=2q-2+\xi.
    \end{align}
    $\mathbf{(b)}$ The total number of line orbits in $\PG(3,q)$ is $2q+7+\xi$.

  \item[(ii)] Let $q$ be odd, $5\le q\le 37$.
   Then under $G_q$, for $\EnG$-lines,  there are
   \begin{align*}
 &  (q-\xi)/3&&\T{ orbits of length }&q^3-q,\db\\
  &  q-1&&\T{ orbits of length }&(q^3-q)/2,\db\\
   &    n_q^{(\xi)} &&\T{ orbits of length }& (q^3-q)/4,
   \end{align*}
   where $n_q^{(1)}=(2q-11)/3,~n_q^{(-1)}=(2q-10)/3,~
    n_q^{(0)}=(2q-6)/3$.\\
     In addition, for $q\in\{7,13,19,25,31,37\}$ where $q\equiv1\pmod3$, there are

       one orbit of length $(q^3-q)/12$ and  two orbits of length $(q^3-q)/3$.

  \item[(iii)] Let $q=8,16,32,64$. Then under $G_q$, for $\EnG$-lines, there are

     $2+\xi$  orbits of length $(q^3-q)/(2+\xi)$ and     $2q-4$ orbits of length $(q^3-q)/2$.
\end{description}
\end{theorem}

\begin{conjecture} \label{conj2:orbEnG} \cite{DMP_OrbLine}
The results of Theorem \ref{th2:MAGMA} hold for all $q\ge5$ with the corresponding parity and $\xi$ value.
\end{conjecture}
For odd $q\not\equiv0\pmod3$, the conjecture on \eqref{eq2:L_EnG} is given also in \cite{GulLav}.

The unions (classes) of line orbits are considered in \cite[Chapter 21]{Hirs_PG3q}; they are called $\OO_i$ and $\OO'_i=\OO_i\A$. In \cite{DMP_OrbLine} (for all $q\ge2$), \cite{BlokPelSzo} (for all $q\ge23$), and \cite{GulLav} (for odd $q\not\equiv0\pmod3$), these classes (apart from $\OO_6$) are investigated; the sizes and the structures of the orbits forming each class are obtained.

Theorem~\ref{th2:orbLine} and Table \ref{tab1} summarize some results from \cite{BlokPelSzo,DMP_OrbLine,GulLav,Hirs_PG3q} useful in this paper.

  \begin{table}[h]
\caption{Unions (classes) $O_i$ and $O'_i=\OO_i\A$ of line orbits under $G_q$ in $\PG(3,q)$, $q\equiv\xi\pmod3$, $q\ge5$.
$\OO_i=\OO'_i, ~i=2,4,6$. $L_{\lambda\mathrm{\Sigma}}^{(\xi)\bullet}$ is the total number of orbits in the class $\OO_\lambda$. $\#\OO_{\lambda_j}$ is the size of the $j$-th orbit of a class $\OO_\lambda$ consisting of 2 or 3 orbits}
\label{tab1}
\centering
\begin{tabular}{llcccccl}\hline\noalign{\smallskip}
&&content&&&&&$\#\OO_{\lambda_1}$\\
$\OO_i$&&of the&size of&&&&$\#\OO_{\lambda_2}$\\
$\OO'_i$&$\OO_\lambda$&class&the class&$\xi$&$L_{\lambda\mathrm{\Sigma}}^{(\xi)\od}$&$L_{\lambda\mathrm{\Sigma}}^{(\xi)\ev}$&$\#\OO_{\lambda_3}$
\\\noalign{\smallskip}\hline\noalign{\smallskip}
$\OO_1$&$\OO_\RC$&$\RC$-lines&$(q^2+q)/2$&any&$1$&$1$&\\
$\OO'_1$&$\OO_\RA$&$\RA$-lines&$(q^2+q)/2$&$\ne0$&$1$&$1$&\\
$\OO_2$&$\OO_\Tr$&$\Tr$-lines&$q+1$&any&$1$&$1$&\\
$\OO_3$&$\OO_\IC$&$\IC$-lines&$(q^2-q)/2$&any&$1$&$1$&\\
$\OO'_3$&$\OO_\IA$&$\IA$-lines&$(q^2-q)/2$&$\ne0$&$1$&$1$&\\
$\OO_4$&$\OO_\UG$&$\UG$-lines&$q^2+q$&any&$1$&$2$&$q+1$\\
&&&&&&&$q^2-1$\\
$\OO_5$&$\OO_\UnG$&$\UnG$-lines&$q^3-q$&any&$2$&$1$&$(q^3-q)/2$\\
&&&&&&&$(q^3-q)/2$\\
$\OO'_5$&$\OO_\EG$&$\EG$-lines&$q^3-q$&$\ne0$&$2$&$1$&$(q^3-q)/2$\\
&&&&&&&$(q^3-q)/2$\\
$\OO_6$&$\OO_\EnG$&$\EnG$-lines&$(q^2-q)(q^2-1)$&any&$L_{\EnG\mathrm{\Sigma}}^{(\xi)\od}$&$L_{\EnG\mathrm{\Sigma}}^{(\xi)\ev}$&\\
$\OO_7$&$\OO_\Ar$&$\Ar$-line&$1$&$0$&$1$&--&\\
$\OO_8$&$\OO_\EA$&$\EA$-lines&$(q+1)(q^2-1)$&$0$&$3$&--&$q^3-q$\\
&&&&&&&$(q^2-1)/2$\\
&&&&&&&$(q^2-1)/2$\\\noalign{\smallskip}\hline
\end{tabular}
\end{table}
In the last column of Table \ref{tab1}, the sizes of the orbits $\OO_{\lambda_j}$ of a class $\OO_\lambda$ consisting of 2 or 3 orbits are given from top to bottom,
e.g. for $\OO_\EA$ we have $\OO_{\EA_1}=q^3-q$, $\OO_{\EA_2}=(q^2-1)/2$, $\OO_{\EA_3}=(q^2-1)/2$.

\begin{theorem}\label{th2:orbLine}  \emph{\cite{BlokPelSzo,DMP_OrbLine,GulLav,Hirs_PG3q}} Let $q\ge5$. The lines of $\PG(3,q)$ can be partitioned into classes called $\OO_i$ and $\OO'_i$, each of which is a union of orbits under $G_q$. The classification of the unions \emph{(}classes\emph{)} of line orbits is given in Table \emph{\ref{tab1}}.

If $q\not\equiv0\pmod3$ we have
\begin{align}
  &\OO'_i=\OO_i\A,~\#\OO'_i=\#\OO_i,~i=1,\ldots,6;~\OO_i=\OO'_i, ~i=2,4,6;\label{eq2:O'=OU}\\
  &\OO_\RA=\OO_\RC\A,~\OO_\IA=\OO_\IC\A,~\OO_\EG=\OO_\UnG\A,~\OO_\lambda=\OO_\lambda\A,\,\lambda\in\{\Tr,\UG,\EnG\}.\nt
     \end{align}
\end{theorem}

In \cite[Theorem 3.2]{DMP_OrbLine}, the cases when Table \ref{tab1} holds for $q=2,3,4$ are noted.

\begin{theorem}\label{th2_null_pol} \emph{\cite[Theorem 4.3]{DMP_OrbLine}}
 Let $q\not\equiv0\pmod 3$. Let $\LL$ be an orbit of lines under~$G_q$. Then $\LL\mathfrak{A}$ also is an orbit of lines under~$G_q$.
\end{theorem}

In \cite{GulLav}, the line orbits are denoted by $\LL_i$ and $\LL^\bot_i=\LL_i\A$, $i=1,\ldots,10$. We give the correspondence between $\LL_i$ and the notations of this paper (in \cite{DMP_OrbLine,DMP_PlLineInc} the notations are the same as in this paper).
\begin{align}\label{eq2:L_Lav}
&\LL_1=\OO'_1=\OO_\RA,~\LL_2=\OO_2=\OO'_2=\OO_\Tr,~\LL_3=\OO_4=\OO'_1=\OO_\UG,\db\\
&\LL_4=\OO'_{5_2}=\OO_{\EG_2},~\LL_5=\OO'_{5_1}=\OO_{\EG_1},~\LL_6=\LL^\bot_1=\OO_1=\OO_\RC,\dbn\\
&\LL_7=\LL^\bot_4=\OO_{5_2}=\OO_{\UnG_2},~
\LL_8=\LL^\bot_5=\OO_{5_1}=\OO_{\UnG_1},~\LL_9=\OO_3=\OO_\IC,\dbn\\
&\LL_{10}=\LL^\bot_9=\OO'_3=\OO_\IA.\nt
\end{align}

\subsection{The plane-line incidence matrix  of $\PG(3,q)$}\label{subsec_incid}
The $\beta_{3,q}\times\theta_{3,q}$ plane-line incidence matrix  $\I^{\mathrm{\Pi}\mathrm{\Lambda}}$ of $\PG(3,q)$ is considered in \cite{DMP_PlLineInc,GulLav}.

In \cite{DMP_PlLineInc}, all $q\ge2$ are considered, including even $q$ and $q\equiv0\pmod3$,
 see \cite[Section 3]{DMP_PlLineInc}, where the results of the paper are summarized.
 In $\I^{\mathrm{\Pi}\mathrm{\Lambda}}$,  columns correspond to planes, rows correspond to lines, and there is an entry  ``1'' if the corresponding line lies in the corresponding plane.  In \cite{DMP_PlLineInc}, $\I^{\mathrm{\Pi}\mathrm{\Lambda}}$ is partitioned into $\#\OO_\lambda\times\#\N_\pi$ submatrices $\I_{\pi,\lambda}^{\mathrm{\Pi}\mathrm{\Lambda}}$, $\lambda\in\Lk^{(\xi)}$, $\pi\in\Pk$. If $L_{\lambda\mathrm{\Sigma}}^{(\xi)\bullet}>1$, see Table~\ref{tab1}, then $\I_{\pi,\lambda}^{\mathrm{\Pi}\mathrm{\Lambda}}$ splits into $L_{\lambda\mathrm{\Sigma}}^{(\xi)\bullet}$ submatrices
$\I_{\pi,\lambda_j}^{\mathrm{\Pi}\mathrm{\Lambda}}$, $j=1,\ldots,L_{\lambda\mathrm{\Sigma}}^{(\xi)\bullet}$.

The values of $\mathrm{\Pi}_{\pi,\lambda}^{(\xi)\bullet}$, $\mathrm{\Lambda}_{\lambda,\pi}^{(\xi)\bullet}$ for all $L_{\lambda\mathrm{\Sigma}}^{(\xi)\bullet}$ and $\mathrm{\Pi}_{\pi,\lambda_j}^{(\xi)\bullet}$, $\mathrm{\Lambda}_{\lambda_j,\pi}^{(\xi)\bullet}$ for $L_{\lambda\mathrm{\Sigma}}^{(\xi)\bullet}=2,3$ are obtained in \cite{DMP_PlLineInc}
and collected in \cite[Tables 1, 2]{DMP_PlLineInc}. For the class $\OO_6$, the values of $\mathrm{\Pi}_{\pi,\lambda}^{(\xi)\bullet}$ are average over all $\EnG$-lines.

In \cite{GulLav}, for odd $q\not\equiv0\pmod3$, the ten orbits $\LL_i$, see \eqref{eq2:L_Lav}, are considered and the corresponding values $\mathrm{\Pi}_{\pi,\lambda}^{(\xi)\od}$ are obtained (they are denoted by $OD_2(\ell)$ and are called ``the plane orbit distribution of a line $\ell$"). These results of \cite{GulLav} are in accordance with the ones of \cite{DMP_PlLineInc}.

\subsection{The point-line incidence matrix  of $\PG(3,q)$}\label{subsec_incid}

\textbf{Notation 2}
According to Notation 1, let $\pk$ and $\lambda$ be the type of a point and of a line and let $\Mk^{(\xi)}$ and $\Lk^{(\xi)}$ be the lists of the possible types. By default,  $\pk\in\Mk^{(\xi)}$, $\lambda\in\Lk^{(\xi)}$. A $\pk$-point is a point of the orbit $\M_\pk$; a $\lambda$-line is a line of the class $\OO_\lambda$.
In addition to Notation 1, the following notation is used:
\begin{align*}
&\Lb_{\lambda_j,\pk}^{(\xi)\bullet}&&\T{the number of lines from an orbit $\OO_{\lambda_j}$ through a $\pk$-point};\db\\
&\Lb_{\lambda,\pk}^{(\xi)\bullet}&&\T{the total number of $\lambda$-lines through a $\pk$-point};\db\\
&\Pb_{\pk,\lambda_j}^{(\xi)\bullet} &&\T{the number of $\pk$-points on a line of an orbit }\OO_{\lambda_j};\db\\
&\Pb_{\pk,\lambda}^{(\xi)\bullet}&&\T{the average number of $\pk$-points on a $\lambda$-line over all the $\lambda$-lines};\db\\
&&&\T{if the class $\OO_\lambda$ consists of \emph{a single orbit} then }\Pb_{\pk,\lambda}^{(\xi)}\T{ is \emph{the exact number}}\db\\
&&&\T{of $\pk$-points on each $\lambda$-line};\db\\
&\I^{\Pb\Lb}&&\T{the $\beta_{3,q}\times\theta_{3,q}$ point-line incidence matrix of }\PG(3,q);\db\\
&\I_{\pk,\lambda}^{\Pb\Lb}&&\T{the $\#\OO_\lambda\times\#\M_\pk$ submatrix of $\I^{\Pb\Lb}$ with incidences between }\db\\
&&&\T{$\pk$-points and $\lambda$-lines}; \db\\
&\I_{\pk,\lambda_j}^{\Pb\Lb}&&\T{the $\#\OO_{\lambda_j}\times\#\M_\pk$ submatrix of $\I_{\pk,\lambda}^{\Pb\Lb}$ with incidences between }\db\\
&&&\T{$\pk$-points and $\lambda_j$-lines.}
 \end{align*}

In  $\I^{\Pb\Lb}$, columns correspond to points, rows correspond to lines, and there is an entry  ``1'' if the corresponding point lies on the corresponding line.  Every column and every row of $\I^{\Pb\Lb}$ contains $\theta_{2,q}$ and $\theta_{1,q}$ ones, respectively, as in $\PG(3,q)$, there are $\theta_{2,q}$ lines through every point and $\theta_{1,q}$ points in every line. Thus, $\I^{\Pb\Lb}$ is a tactical configuration \cite[Chapter 2.3]{Hirs_PGFF}, \cite[Chapter 7, Section~2]{Lidl_Nied}.
Moreover, $\I^{\Pb\Lb}$  gives a 2-$(\theta_{3,q},\theta_{1,q},1)$ design~\cite{HandbCombDes2v_k_lamb} since there is exactly one line through any two points.

\begin{definition}\label{def2_config}\cite{GroppConfig}
    A configuration $(v_r,b_k)$  is an incidence structure of $v$ points and $b$ lines such that
 each line contains $k$ points, each point lies on $r$ lines, and
 two different points are connected by at most one line. If $v = b$ and, hence, $r = k$, the configuration is symmetric, denoted by $v_k$.
\end{definition}
\noindent For an introduction to configurations see \cite{DFGMP_SymConf,GroppConfig} and the references therein.

The transposition $(\I^{\Pb\Lb})^{tr}$ gives the $\theta_{3,q}\times\beta_{3,q}$ line-point incidence matrix. It can be viewed as a $(v_r,b_k)$ configuration  with $v=\beta_{3,q}$, $b=\theta_{3,q}$, $r=\theta_{1,q}$, $k=\theta_{2,q}$, as there is at most one point as the intersection of two different lines.

In \cite{GulLav}, for odd $q\not\equiv0\pmod3$, the ten orbits $\LL_i$, see \eqref{eq2:L_Lav}, are considered and the corresponding values $\Pb_{\pk,\lambda}^{(\xi)\od}$, $\Pb_{\pk,\lambda_j}^{(\xi)\od}$ are obtained (they are denoted by $OD_0(\ell)$ and are called ``the point orbit distribution of a line $\ell$").

These results of \cite{GulLav} are obtained also in this paper by another way. In this process we also obtained the values of $\Pb_{\pk,\lambda}^{(\xi)\ev}$, $\Pb_{\pk,\lambda_j}^{(\xi)\ev}$ for even $q\not\equiv0\pmod3$ and the values of $\Lb_{\lambda,\pk}^{(\xi)\bullet}$, $\Lb_{\lambda_j,\pk}^{(\xi)\bullet}$ for all odd and even $q\not\equiv0\pmod3$, see the first two tables of Section \ref{sec_mainres} and Section~\ref{sec:results_q_ne0}.

Moreover, in this paper we obtained also values $\Pb_{\pk,\lambda}^{(0)\od}$, $\Pb_{\pk,\lambda_j}^{(0)\od}$, $\Lb_{\lambda,\pk}^{(0)\od}$, and $\Lb_{\lambda_j,\pk}^{(0)\od}$ for $q\equiv0\pmod3$, see the last two tables of Section \ref{sec_mainres} and Section \ref{sec:results_q=0}.

As we mentioned above, for the class $\OO_6$, in this paper only average and cumulative results are obtained.

\section{The main results}\label{sec_mainres}
\begin{remark}
We call $\Pb_{\pk,\lambda}^{(\xi)\bullet}$  \emph{the average number} of $\pk$-points on a $\lambda$-line over all the $\lambda$-lines. If the class  $\OO_\lambda$ of $\lambda$-lines consists of a single orbit, i.e. $L_{\lambda\mathrm{\Sigma}}^{(\xi)\bullet}=1$, then $\Pb_{\pk,\lambda}^{(\xi)\bullet}$  is \emph{the exact number} of $\pk$-points on each $\lambda$-line, see Lemma \ref{lemma4_line&point}. The situation is always clear by the context.
If $L_{\lambda\mathrm{\Sigma}}^{(\xi)\bullet}=1$ then $\Pb_{\pk,\lambda}^{(\xi)\bullet}$ certainly is an integer. If $\lambda$-lines form two or more orbits, i.e. $L_{\lambda\mathrm{\Sigma}}^{(\xi)\bullet}\ge2$, then $\Pb_{\pk,\lambda}^{(\xi)\bullet}$ may be not  integer as well as an integer.

 On the other hand, regardless of the number of orbits in $\OO_\lambda$, for all pairs $(\pk,\lambda)$, we always have the same total number of $\lambda$-lines through each $\pk$-point, i.e. $\mathrm{\Lambda}_{\lambda,\pk}^{(\xi)\bullet}$ always is an integer,  again see Lemma \ref{lemma4_line&point}.
\end{remark}

From now on, we consider $q\ge5$, apart from Theorem \ref{th3:q=2 3 4}. Theorem~\ref{th3:q=2 3 4} is obtained by an  exhaustive computer search using the computer algebra system Magma \cite{Magma}.

Tables \ref{tab2}--\ref{tab5} and Theorem \ref{th3_main_res} summarize the results of  Sections \ref{sec:useful}--\ref{sec:gen res} for $q\ge5$.

\begin{table}[htbp]
\caption{Values $\Pb_{\pk,\lambda}^{(\xi)}$ (top entry) and $\Lb_{\lambda,\pk}^{(\xi)}$ (bottom entry) for submatrices $\I_{\pk,\lambda}^{\Pb\Lb}$ of the point-line incidence matrix of $\PG(3,q)$, $q\equiv\xi\pmod3$, $\xi\in\{1,-1\}$, $q\ge5$, $\pk\in\Mk^{(\ne0)}$, $\lambda\in\Lk^{(\ne0)}$. The superscript~$(\xi)$ is $(\ne0)$ if a value is the same for all~$q\not\equiv0\pmod3$}
\label{tab2}
\centering
\begin{tabular}{lccccccc}\hline
&&&$\M_1^{(\ne0)}$&$\M_2^{(\ne0)}$&$\M_3^{(\ne0)}$&$\M_4^{(\ne0)}$&$\M_5^{(\ne0)}$\\
&&&$\C\T{-}$&$\Tr\T{-}$&$3_\mathrm{\Gamma}\T{-}$&$1_\mathrm{\Gamma}\T{-}$&$0_\mathrm{\Gamma}\T{-}$\\
$\OO_j$&$\lambda\T{-lines}$&$\Pb_{\pk,\lambda}^{(\xi)}$&points&points&points&points&points\\
$\OO'_j$&$\#\OO_\lambda$&$\Lb^{(\xi)}_{\lambda,\pk}$&$q+1$&$q^2+q$&$\frac{1}{6}(q^3-q)$&$ \frac{1}{2}(q^3-q)$&$\frac{1}{3}(q^3-q)$\\\hline

$\OO_1$&$\RC\T{-lines}$&$\Pb_{\pk,\RC}^{(1)}$&$2$&$0$&$\frac{1}{3}(q-1)$&$0$&$\frac{2}{3}(q-1)$\\
&$\frac{1}{2}(q^2+q)$&$\Lb^{(1)}_{\RC,\pk}$&$q$&$0$&$ 1$&$ 0$&$ 1$\\\hline

$\OO_1$&$\RC\T{-lines}$&$\Pb_{\pk,\RC}^{(-1)}$&$2$&$0$&$0$&$ q-1$&$0$\\
&$\frac{1}{2}(q^2+q)$&$\Lb^{(-1)}_{\RC,\pk}$&$q$&$0$&$ 0$&$ 1$&$ 0$\\\hline

$\OO'_1$&$\RA\T{-lines}$&$\Pb_{\pk,\RA}^{(\ne0)}$&$0$&$2$&$q-1$&$0$&$0$\\
$$&$\frac{1}{2}(q^2+q)$&$\Lb_{\RA,\pk}^{(\ne0)}$&$0$&$1$&$3$&$0$&$0$\\\hline

$\OO_2$&$\Tr\T{-lines}$&$\Pb_{\pk,\Tr}^{(\ne0)}$&$1$&$q$&$0$&$0$&$0$\\
$\OO'_2$&$q+1$&$\Lb_{\Tr,\pk}^{(\ne0)}$&$1$&$1$&$0$&$0$&$0$\\\hline

$\OO_3$&$\IC\T{-lines}$&$\Pb^{(1)}_{\pk,\IC}$&$0$&$0$&$ 0$&$ q+1$&$ 0$\\
&$\frac{1}{2}(q^2-q)$&$\Lb^{(1)}_{\IC,\pk}$&$0$&$0$&$ 0$&$ 1$&$ 0$\\\hline

$\OO_3$&$\IC\T{-lines}$&$\Pb^{(-1)}_{\pk,\IC}$&$0$&$0$&$\frac{1}{3}(q+1)$&$ 0$&$\frac{2}{3}(q+1)$\\
$$&$\frac{1}{2}(q^2-q)$&$\Lb^{(-1)}_{\IC,\pk}$&$0$&$0$&$ 1$&$ 0$&$ 1$\\\hline

$\OO'_3$&$\IA\T{-lines}$&$\Pb_{\pk,\IA}^{(\ne0)}$&$0$&$0$&$0$&$q+1$&$0$\\
&$\frac{1}{2}(q^2-q)$&$\Lb_{\IA,\pk}^{(\ne0)}$&$0$&$0$&$0$&$1$&$0$\\\hline

$\OO_4$&$\UG\T{-lines}$&$\Pb_{\pk,\UG}^{(\ne0)}$&$1$&$1$&$\frac{1}{2}(q-1)$&$\frac{1}{2}(q-1)$&$0$\\
$\OO'_4$&$q^2+q$&$\Lb_{\UG,\pk}^{(\ne0)}$&$q$&$1$&$3$&$1$&$0$\\\hline

$\OO_5$&$\UnG\T{-lines}$&$\Pb^{(1)}_{\pk,\UnG}$&$1$&$1$&$ \frac{1}{6}(q-4)$&$ \frac{1}{2}q$&$\frac{1}{3} (q-1)$\\
&$q^3-q$&$\Lb^{(1)}_{\UnG,\pk}$&$q^2-q$&$q-1$&$ q-4$&$ q$&$ q-1$\\\hline

$\OO_5$&$\UnG\T{-lines}$&$\Pb^{(-1)}_{\pk,\UnG}$&$1$&$1$&$\frac{1}{6}(q-2)$&$ \frac{1}{2}(q-2)$&$\frac{1}{3}(q+1)$\\
&$q^3-q$&$\Lb^{(-1)}_{\UnG,\pk}$&$q^2-q$&$q-1$&$ q-2$&$ q-2$&$ q+1$\\\hline

$\OO'_5$&$\EG\T{-lines}$&$\Pb_{\pk,\EG}^{(\ne0)}$&$0$&$2$&$ \frac{1}{2}(q-2)$&$ \frac{1}{2}q$&$0$\\
&$q^3-q$&$\Lb_{\EG,\pk}^{(\ne0)}$&$0$&$2(q-1)$&$ 3(q-2)$&$ q$&$0$\\\hline

$\OO_6$&$\EnG\T{-lines}$&$\Pb^{(1)}_{\pk,\EnG}$&$0$&$1$&$ \frac{q^2-3q+4}{6(q-1)}$&$ \frac{(q+1)(q-2)}{2(q-1)}$&$ \frac{q^2+1}{3(q-1)}$\\
$\OO'_6$&$(q^2-q)\cdot$&$\Lb^{(1)}_{\EnG,\pk}$&$0$&$(q-1)^2$&$q^2- $&$(q+1)\cdot$&$q^2+1$\\
&$(q^2-1)$&&&&$3q+4$&$(q-2)$&\\\hline

$\OO_6$&$\EnG\T{-lines}$&$\Pb^{(-1)}_{\pk,\EnG}$&$0$&$1$&$ \frac{1}{6}(q-2)$&$ \frac{1}{2}q$&$ \frac{1}{3}(q+1)$\\
$\OO'_6$&$(q^2-q)\cdot$&$\Lb^{(-1)}_{\EnG,\pk}$&$0$&$(q-1)^2$&$(q-1)\cdot$&$ q^2-q$&$ q^2-1$\\
&$(q^2-1)$&&&&$(q-2)$&&\\\hline
\end{tabular}
\end{table}

\begin{table}[htbp]
\caption{Values $\Pb_{\pk,\lambda_j}^{(\xi)\bullet}$ (top entry) and
 $\Lb_{\lambda_j,\pk}^{(\xi)\bullet}$ (bottom entry) for submatrices $\I_{\pk,\lambda_j}^{\Pb\Lb}$ of the point-line incidence matrix of $\PG(3,q), q\ge5,~\pk\in\Mk^{(\ne0)}$; $j=1,2$; $\lambda=\UG$ with even $q\not\equiv0\pmod3$ ($\UG_1$- and $\UG_2$-lines);
 $\lambda=\UnG$ with odd $q\equiv\xi\pmod3$, $\xi\in\{1,-1\}$ ($\UnG_1$- and $\UnG_2$-lines for $\xi=1$ and $\xi=-1$);
 $\lambda=\EG$ with odd $q\not\equiv0\pmod3$ ($\EG_1$- and $\EG_2$-lines)}
\label{tab3}
 \centering
\begin{tabular}{lccccccc}\hline
&&&$\M_1^{(\ne0)}$&$\M_2^{(\ne0)}$&$\M_3^{(\ne0)}$&$\M_4^{(\ne0)}$&$\M_5^{(\ne0)}$\\
&&&$\C\T{-}$&$\Tr\T{-}$&$3_\mathrm{\Gamma}\T{-}$&$1_\mathrm{\Gamma}\T{-}$&$0_\mathrm{\Gamma}\T{-}$\\
$\OO_{i_j}$&$\lambda_j\T{-lines}$&$\Pb_{\pk,\lambda_j}^{(\xi)}$&$\T{points}$&$\T{points}$&$\T{points}$&$\T{points}$&$\T{points}$\\
$\OO'_{i_j}$&$\#\OO_{\lambda_j}$&$\Lb^{(\xi)}_{\lambda_j,\pk}$&$q+1$&$q^2+q$&$\frac{q^3-q}{6}$&$ \frac{q^3-q}{2}$&$\frac{q^3-q}{3}\vphantom{H_{H_H}}$\\\hline

$\OO_{4_1}$&$\UG_1\T{-lines}$&$\Pb_{\pk,\UG_1}^{(\ne0)\mathrm{ev}}$&$1$&$q$&$ 0$&$ 0$&$0$\\
&$q+1$&$\Lb_{\UG_1,\pk}^{(\ne0)\mathrm{ev}}$&$1$&$1$&$ 0$&$ 0$&$0$\\\hline

$\OO_{4_2}$&$\UG_2\T{-lines}$&$\Pb_{\pk,\UG_2}^{(\ne0)\mathrm{ev}}$&$1$&$0$&$ \frac{1}{2}q$&$\frac{1}{2}q$&$0$\\
&$q^2-1$&$\Lb_{\UG_2,\pk}^{(\ne0)\mathrm{ev}}$&$q-1$&$0$&$ 3$&$ 1$&$0$\\\hline

$\OO_{5_1}$&$\UnG_1\T{-}$&$\Pb^{(1)\mathrm{od}}_{\pk,\UnG_1}$&$1$&$0$&$ \frac{1}{6}(q-1)$&$ \frac{1}{2}(q+1)$&$ \frac{1}{3}(q-1)$\\
$\xi=$&lines&&&&&&\\
$1$&$\frac{1}{2}(q^3-q)$&$\Lb^{(1)\mathrm{od}}_{\UnG_1,\pk}$&$\frac{1}{2}(q^2-q)$&$0
$&$\frac{1}{2}(q-1)$&$\frac{1}{2}(q+1)$&$\frac{1}{2}(q-1)$\\\hline

 $\OO_{5_2}$&$\UnG_2\T{-}$&$\Pb^{(1)\mathrm{od}}_{\pk,\UnG_2}$&$1$&$2$&$\frac{1}{6}(q-7)$&$\frac{1}{2}(q-1)$&$\frac{1}{3}(q-1)$\\
 $\xi=$&lines&&&&&&\\
 $1$&$\frac{1}{2}(q^3-q)$&$\Lb^{(1)\mathrm{od}}_{\UnG_2,\pk}$&$\frac{1}{2}(q^2-q)$&$q-1
 $&$\frac{1}{2}(q-7)$&$\frac{1}{2}(q-1)$&$\frac{1}{2}(q-1)$\\\hline

 $\OO_{5_1}$&$\UnG_1\T{-}$&$\Pb^{(-1)\mathrm{od}}_{\pk,\UnG_1}$&$1$&$0$&$ \frac{1}{6}(q+1)$&$ \frac{1}{2}(q-1)$&$ \frac{1}{3}(q+1)$\\
 $\xi=$&lines&&&&&&\\
 $-1$&$\frac{1}{2}(q^3-q)$&$\Lb^{(-1)\mathrm{od}}_{\UnG_1,\pk}$&$\frac{1}{2}(q^2-q)$&$0
 $&$\frac{1}{2}(q+1)$&$\frac{1}{2}(q-1)$&$\frac{1}{2}(q+1)$\\\hline

 $\OO_{5_2}$&$\UnG_2\T{-}$&$\Pb^{(-1)\mathrm{od}}_{\pk,\UnG_2}$&$1$&$2$&$\frac{1}{6}(q-5)$&$\frac{1}{2}(q-3)$&$ \frac{1}{3}(q+1)$\\
 $\xi=$&lines&&&&&&\\
 $-1$&$\frac{1}{2}(q^3-q)$&$\Lb^{(-1)\mathrm{od}}_{\UnG_2,\pk}$&$\frac{1}{2}(q^2-q)$&$q-1
 $&$\frac{1}{2}(q-5)$&$\frac{1}{2}(q-3)$&$ \frac{1}{2}(q+1)$\\\hline

 $\OO'_{5_1}$&$\EG_1\T{-lines}$&$\Pb_{\pk,\EG{_1}}^{(\ne0)\mathrm{od}}$&$0$&$1$&$ \frac{1}{2}(q-1)$&$ \frac{1}{2}(q+1)$&$0$\\
 &$\frac{1}{2}(q^3-q)$&$\Lb_{\EG{_1},\pk}^{(\ne0)\mathrm{od}}$&$0$&$\frac{1}{2}(q-1)$&$ \frac{3}{2}(q-1)$&$ \frac{1}{2}(q+1)$&$0$\\\hline

 $\OO'_{5_2}$&$\EG_2\T{-lines}$&$\Pb_{\pk,\EG{_2}}^{(\ne0)\mathrm{od}}$&$0$&$3$&$ \frac{1}{2}(q-3)$&$\frac{1}{2}( q-1)$&$0$\\
 &$\frac{1}{2}(q^3-q)$&$\Lb_{\EG{_2},\pk}^{(\ne0)\mathrm{od}}$&$0$&$\frac{3}{2}(q-1)$&$ \frac{3}{2}(q-3)$&$ \frac{1}{2}(q-1)$&$0$\\\hline
\end{tabular}
\end{table}

\begin{table}[htbp]
\caption{Values $\Pb_{\pk,\lambda}^{(0)}$ (top entry) and $\Lb_{\lambda,\pk}^{(0)}$ (bottom entry) for submatrices $\I_{\pk,\lambda}^{\Pb\Lb}$ of the point-line incidence matrix $\I^{\Pb\Lb}$ of $\PG(3,q)$, $q\equiv0\pmod3$, $q\ge9$, $\lambda\in\Lk^{(0)}$, $\pk\in\Mk^{(0)}$}
\label{tab4}
\centering
\begin{tabular}{cccccccc}\hline
&&&$\M_1^{(0)}$&$\M_2^{(0)}$&$\M_3^{(0)}$&$\M_4^{(0)}$&$\M_5^{(0)}$\\
&&&$\C\T{-}$&$(q+1)_\mathrm{\Gamma}\T{-}$&$\TO\T{-}$&$\RC\T{-}$&$\IC\T{-}$\\
&$\lambda\T{-lines}$&$\Pb_{\pk,\lambda}^{(0)}$&$\,\T{points}$&$\T{points}$&$\T{points}$&$\T{points}$&$\T{points}$\\
$\OO_j$&$\#\OO_\lambda$&$\Lb^{(0)}_{\lambda,\pk}$&$q+1$&$
q+1$&$q^2-1$&$ \frac{1}{2}(q^3-q)$&$\frac{1}{2}(q^3-q)$\\\hline

$\OO_1$&$\RC\T{-lines}$&$\Pb^{(0)}_{\pk,\RC}$&$2$&$0$&$0$&$q-1$&$0$\\
&$\frac{1}{2}(q^2+q)$&$\Lb^{(0)}_{\RC,\pk}$&$q$&$0$&$0$&$1$&$0$\\\hline

$\OO_2$&$\Tr\T{-lines}$&$\Pb^{(0)}_{\pk,\Tr}$&$1$&$1$&$q-1$&$0$&$0$\\
&$q+1$&$\Lb^{(0)}_{\Tr,\pk}$&$1$&$1$&$1$&$0$&$0$\\\hline

$\OO_3$&$\IC\T{-lines}$&$\Pb^{(0)}_{\pk,\IC}$&$0$&$0$&$0$&$0$&$q+1$\\
&$\frac{1}{2}(q^2-q)$&$\Lb^{(0)}_{\IC,\pk}$&$0$&$0$&$0$&$0$&$1$\\\hline

$\OO_4$&$\UG\T{-lines}$&$\Pb^{(0)}_{\pk,\UG}$&$1$&$1$&$0$&$\frac{1}{2}(q-1)$&$\frac{1}{2}(q-1)$\\
&$q^2+q$&$\Lb^{(0)}_{\UG,\pk}$&$q$&$q$&$0$&$1$&$1$\\\hline

$\OO_5$&$\UnG\T{-lines}$&$\Pb^{(0)}_{\pk,\UnG}$&$1$&$0$&$1$&$\frac{1}{2}(q-2)$&$\frac{1}{2}q$\\
&$q^3-q$&$\Lb^{(0)}_{\UnG,\pk}$&$q^2-q$&$0$&$q$&$q-2$&$q$\\\hline

$\OO_6$&$\EnG\T{-lines}$&$\Pb^{(0)}_{\pk,\EnG}$&$0$&$0$&$1$&$\frac{q^2-q+1}{2(q-1)}$&$\frac{q^2-q-1}{2(q-1)}$\\
&$(q^2-q)\cdot$&$\Lb^{(0)}_{\EnG,\pk}$&$0$&$0$&$q^2-q$&$q^2-q+1$&$q^2-q-1$\\
&$(q^2-1)$&&&&&\\\hline

$\OO_7$&$\Ar\T{-lines}$&$\Pb_{\pk,\Ar}^{(0)}$&$0$&$q+1$&$0$&$0$&$0$\\
&$1$&$\Lb^{(0)}_{\Ar,\pk}$&$0$&$1$&$0$&$0$&$0$\\\hline

$\OO_8$&$\EA\T{-lines}$&$\Pb_{\pk,\EA}^{(0)}$&$0$&$1$&$\frac{q}{q+1}$&$\frac{q^2}{2(q+1)}$&$\frac{q^2}{2(q+1)}$\\
&$(q+1)\cdot$&$\Lb^{(0)}_{\EA,\pk}$&$0$&$q^2-1$&$q$&$q$&$q$\\
&$(q^2-1)$&&&&\\\hline
\end{tabular}
\end{table}

\begin{table}[h]
\caption{Values $\Pb_{\pk,\lambda_j}^{(0)}$ (top entry) and $\Lb_{\lambda_j,\pk}^{(0)}$ (bottom entry) for submatrices $\I_{\pk,\lambda_j}^{\Pb\Lb}$ of the point-line incidence matrix $\I^{\Pb\Lb}$ of $\PG(3,q)$, $q\equiv0\pmod3$, $q\ge9$, $\pk\in\Mk^{(0)}$, $\lambda\in\{\UnG,\EA\}$,  $j=1,2$ if $\lambda=\UnG$, $j=1,2,3$ if $\lambda=\EA $}
\label{tab5}
\centering
\begin{tabular}{cccccccc}\hline
&&&$\M_1^{(0)}$&$\M_2^{(0)}$&$\M_3^{(0)}$&$\M_4^{(0)}$&$\M_5^{(0)}$\\
&&&$\C\T{-}$&$(q+1)_\mathrm{\Gamma}\T{-}$&$\TO\T{-}$&$\RC\T{-}$&$\IC\T{-}$\\
&$\lambda_j\T{-lines}$&$\Pb_{\pk,\lambda}^{(0)}$&$\,\T{points}$&$\T{points}$&$\T{points}$&$\T{points}$&$\T{points}$\\
$\OO_{i_j}$&$\#\OO_{\lambda_j}$&$\Lb^{(0)}_{\lambda,\pk}$&$q+1$&$
q+1$&$q^2-1$&$\frac{q^3-q}{2}$&$\frac{q^3-q}{2}$\\\hline

$\OO_{5_1}$&$\UnG_1\T{-}$&$\Pb^{(0)}_{\pk,\UnG_1}$&$1$&$0$&$0$&$\frac{1}{2}(q-1)$&$\frac{1}{2}(q+1)$\\
$\xi=$&lines&&&&&&\\
$0$&$\frac{1}{2}(q^3-q)$&$\Lb^{(0)}_{\UnG_1,\pk}$&$\frac{1}{2}(q^2-q)$&$0$&$0$&$\frac{1}{2}(q-1)$&$\frac{1}{2}(q+1)$\\\hline

$\OO_{5_2}$&$\UnG_2\T{-}$&$\Pb^{(0)}_{\pk,\UnG_2}$&$1$&$0$&$2$&$\frac{1}{2}(q-3)$&$\frac{1}{2}(q-1)$\\
$\xi=$&lines&&&&&&\\
$0$&$\frac{1}{2}(q^3-q)$&$\Lb^{(0)}_{\UnG_2,\pk}$&$\frac{1}{2}(q^2-q)$&$0$&$q$&$\frac{1}{2}(q-3)$&$\frac{1}{2}(q-1)$\\\hline

$\OO_{8_1}$&$\EA_1\T{-lines}$&$\Pb_{\pk,\EA_1}^{(0)}$&$0$&$1$&$1$&$\frac{1}{2}(q-1)$&$\frac{1}{2}(q-1)$\\
&$q^3-q$&$\Lb^{(0)}_{\EA_1,\pk}$&$0$&$q^2-q$&$q$&$q-1$&$q-1$\\\hline

$\OO_{8_2}$&$\EA_2\T{-lines}$&$\Pb_{\pk,\EA_2}^{(0)}$&$0$&$1$&$0$&$q$&$0$\\
&$\frac{1}{2}(q^2-1)$&$\Lb^{(0)}_{\EA_2,\pk}$&$0$&$\frac{1}{2}(q-1)$&$0$&$1$&$0$\\\hline

$\OO_{8_3}$&$\EA_3\T{-lines}$&$\Pb_{\pk,\EA_3}^{(0)}$&$0$&$1$&$0$&$0$&$q$\\
&$\frac{1}{2}(q^2-1)$&$\Lb^{(0)}_{\EA_3,\pk}$&$0$&$\frac{1}{2}(q-1)$&$0$&$0$&$1$\\\hline
\end{tabular}
\end{table}

For the point-line incidence matrix $\I^{\Pb\Lb}$ of $\PG(3,q)$,
 $q\equiv\xi\pmod3$, Tables~\ref{tab2} (for $q\not\equiv0\pmod3$) and \ref{tab4} (for $q\equiv0\pmod3$) show the values $\Pb_{\pk,\lambda}^{(\xi)}$  (top entry) and $\Lb_{\lambda,\pk}^{(\xi)}$ (bottom entry) for each pair $(\pk,\lambda)$, $\pk\in\Mk^{(\xi)}$, $\lambda\in\Lk^{(\xi)}$,  where $\Pb_{\pk,\lambda}^{(\xi)}$  is  the exact (if $L_{\lambda\mathrm{\Sigma}}^{(\xi)\bullet}=1$) or average (if $L_{\lambda\mathrm{\Sigma}}^{(\xi)\bullet}\ge2$) number of $\pk$-points on every $\lambda$-line, whereas $\Lb_{\lambda,\pk}^{(\xi)}$  always is the exact number of $\lambda$-lines through every $\pk$-point. In other words, $\Pb_{\pk,\lambda}^{(\xi)}$  is the exact or average number of ones in every row of the submatrix $\I^{\Pb\Lb}_{\pk,\lambda}$  of $\I^{\Pb\Lb}$, whereas $\Lb_{\lambda,\pk}^{(\xi)}$ always is the exact number of ones in every column of $\I^{\Pb\Lb}_{\pk,\lambda}$. In Table \ref{tab2}, the superscript $(\xi)$ is $(\ne0)$ for $\lambda\in\{\RA,\Tr,\IA,\UG,\EG\}$ where the values $\Pb_{\pk,\lambda}^{(\xi)}$, $\Lb_{\lambda,\pk}^{(\xi)}$ are the same for all~$q\not\equiv0\pmod3$.

The total number of orbits of $\lambda$-lines  is given in Table 1.

In Table \ref{tab3}, the values  $\Pb_{\pk,\lambda_j}^{(\xi)\bullet}$ and $\Lb_{\lambda_j,\pk}^{(\xi)\bullet}$ are given for the following cases:
$q\ge5,~\pk\in\Mk^{(\ne0)}$; $\lambda=\UG$ with even $q\not\equiv0\pmod3$ ($\UG_1$- and $\UG_2$-lines);
 $\lambda=\UnG$ with odd $q\equiv\xi\pmod3$, $\xi\in\{1,-1\}$ ($\UnG_1$- and $\UnG_2$-lines for $\xi=1$ and $\xi=-1$);
 $\lambda=\EG$ with odd $q\not\equiv0\pmod3$ ($\EG_1$- and $\EG_2$-lines).

In Table \ref{tab5}, the values  $\Pb_{\pk,\lambda_j}^{(0)}$ and $\Lb_{\lambda_j,\pk}^{(0)}$ are given for the following cases:
$q\equiv0\pmod3$, $q\ge9$, $\pk\in\Mk^{(0)}$, $\lambda\in\{\UnG,\EA\}$,  $j=1,2$ if $\lambda=\UnG$, $j=1,2,3$ if $\lambda=\EA $.

\begin{theorem}\label{th3_main_res}
Let $q\ge5$, $q\equiv\xi\pmod3$. Let notations be as in Section $\ref{sec_prelimin}$ and  Notations~$1, 2$. The following holds:
\begin{description}
  \item[(i)] In $\PG(3,q)$, for the submatrices $\I^{\Pb\Lb}_{\pk,\lambda}$ of the point-line incidence matrix $\I^{\Pb\Lb}$, the values $\Pb_{\pk,\lambda}^{(\xi)}$ (i.e. the exact or average number of $\pk$-points on a $\lambda$-line) and $\Lb_{\lambda,\pk}^{(\xi)}$ (i.e. the exact number of  $\lambda$-lines through a $\pk$-point) are given in Table $\ref{tab2}$ (for $\xi\ne0$) and Table $\ref{tab4}$  (for $\xi=0$).

      For the submatrices $\I^{\Pb\Lb}_{\pk,\lambda_j}$ corresponding to each of two orbits of the classes $\OO_4=\OO_\UG$, $\OO_5=\OO_\UnG $, and $\OO'_5=\OO_\EG$, the values $\Pb_{\pk,\lambda_j}^{(\xi)\bullet}$, $\Lb_{\lambda_j,\pk}^{(\xi)\bullet}$ are given in Table $\ref{tab3}$  (for $\xi\ne0$). For the submatrices $\I^{\Pb\Lb}_{\pk,\lambda_j}$ corresponding to each of two orbits of the class $\OO_5=\OO_\UnG $ and to each of three orbits of the class $\OO_8=\OO_\EA$, the values $\Pb_{\pk,\lambda_j}^{(0)}$, $\Lb_{\lambda_j,\pk}^{(0)}$ are given in Table $\ref{tab5}$  (for $\xi=0$).

  \item[(ii)] Let a class $\OO_\lambda$ consist of a single
  orbit according to Table $\ref{tab1}$. Then, in Tables $\ref{tab2}$ and $\ref{tab4}$,  the values of $\Pb_{\pk,\lambda}^{(\xi)}$, $\pk\in\Mk^{(\xi)}$, are the \emph{exact numbers} of $\pk$-points on every $\lambda$-line.

\item[(iii)]  Let $q\equiv1\pmod3$. Let $V^{(1)}=\{\OO_1=\OO_\RC,\OO_2=\OO_\Tr,\OO'_3=\OO_\IA\}$. Then, cf. Theorem $\ref{th2_Hirs}$(iv), no two lines of $V^{(1)}$ meet off $\C$. Every point off $\C$ lies on exactly one line of~$V^{(1)}$.

 \item[(iv)]
  Let $q\equiv0\pmod3$. Let $\W^{(0)}=\{\OO_2=\OO_\Tr,\OO_4=\OO_\UG\}$. Let $\mathbb{M}=\C\cup\Ar$-line be the union of the twisted cubic and the $\Ar$-line.  Then
  no two lines of $\W^{(0)}$ meet off $\mathbb{M}$.
 Every point off $\mathbb{M}$ lies on exactly one line of~$\W^{(0)}$, cf. Theorems $\ref{th2_Hirs}$(iv) and $\ref{th3_main_res}$(iii).

\item[(v)] Let $\pk\in\Mk^{(\xi)}$. Let a class $\OO_\lambda$ consist of a single orbit.
Then the submatrix $\I^{\Pb\Lb}_{\pk,\lambda}$ of  $\I^{\Pb\Lb}$ is
  a $(v_r,b_k)$ configuration of Definition \emph{\ref{def2_config}} with $v=\#\M_\pk$, $b=\#\OO_\lambda$, $r=\Lb_{\lambda,\pk}^{(\xi)}$, $k=\Pb_{\pk,\lambda}^{(\xi)}$. Also, up to rearrangement of rows and columns, the submatrices $\I^{\Pb\Lb}_{\pk,\lambda}$ with $\Lb_{\lambda,\pk}^{(\xi)}=1$ can be viewed as a concatenation of $\Pb_{\pk,\lambda}^{(\xi)}$ identity matrices of order $\#\OO_\lambda$. The same holds for the submatrices~$\I^{\Pb\Lb}_{\pk,\lambda_j}$.

\item[(vi)] Let $(\lambda,\pk)\in\{(\UG,\C), (\UnG,\C)\}$ if $\xi\ne0$, and
 $(\lambda,\pk)\in\{(\UnG,\C),$ $ (\EA,(q+1)_\mathrm{\Gamma})\}$ if $\xi=0$. Then,  independently of the number of orbits in the class $\OO_\lambda$, we have exactly one $\pk$-point on every $\lambda$-line. Up to rearrangement of rows and columns, the submatrices $\I^{\Pb\Lb}_{\pk,\lambda}$ can be viewed as a vertical concatenation of\/ $\Lb_{\lambda,\pk}^{(\xi)}$ identity matrices of order $\#\M_\pk$.
\end{description}
\end{theorem}

\begin{theorem}\label{th3:q=2 3 4}
 Let the types of lines and points be as in Tables $\ref{tab1}-\ref{tab5}$.
\begin{description}
    \item[(i)] Let $q=2$. The group $G_2\cong\mathbf{S}_3\mathbf{Z}_2^3$ contains $8$ subgroups isomorphic to $PGL(2,2)$ divided into two conjugacy classes. For one of these subgroups, the matrices corresponding to the projectivities of the subgroup assume the form described by \eqref{eq2_M}. For line and point orbits under this subgroup (and only under it) the point-line incidence matrix has the form of  Tables $\ref{tab2}$ and~$\ref{tab3}$ for even $q\equiv-1\pmod3$ and also Table $\ref{tab1}$ holds.

    \item[(ii)] Let $q=3$. The group $G_3\cong\mathbf{S}_4\mathbf{Z}_2^3$ contains $24$ subgroups isomorphic to $PGL(2,3)$ divided into four conjugacy classes. For one of these subgroups, the matrices corresponding to the projectivities of the subgroup assume the form described by \eqref{eq2_M}. For line and point orbits under this subgroup (and only under it) the point-line incidence matrix has the form of  Tables $\ref{tab4}$ and~$\ref{tab5}$ for $q\equiv0\pmod3$ and also Table $\ref{tab1}$ holds.

    \item[(iii)] Let $q=4$. The group $G_4\cong\mathbf{S}_5\cong P\mathrm{\Gamma} L(2,4)$ contains one subgroup isomorphic to $PGL(2,4)$. The matrices corresponding to the projectivities of this subgroup assume the form described by \eqref{eq2_M} and for line and point orbits under this subgroup the point-line incidence matrix has the form of  Tables $\ref{tab2}$ and $\ref{tab3}$ for even $q\equiv1\pmod3$ and also Table $\ref{tab1}$ holds.

    \item[(iv)] For line orbits under the subgroups of $G_q$ noted in the points (i)--(iii) of this theorem, Theorem $\ref{th2:MAGMA}$ holds also if $q=2,3,4$.
\end{description}
\end{theorem}

\section{Some useful relations}\label{sec:useful}
In this section, we omit the superscripts ``$(\xi)$'', ``od'', and ``ev'' as they are the same for all terms in a formula; in particular, we use $\Lk$ and $L_{\lambda\mathrm{\Sigma}}$ instead of $\Lk^{(\xi)}$ and $L_{\lambda\mathrm{\Sigma}}^{(\xi)\od}$, $L_{\lambda\mathrm{\Sigma}}^{(\xi)\ev}$. In the rest of the paper, when relations of this section are applied, we add the superscripts if they are necessary by the context.

\begin{lemma}\label{lemma4_line&point}
The following holds:
\begin{description}
  \item[(i)]
  The number  $\Lb_{\lambda_j,\pk}$ of lines from an orbit $\OO_{\lambda_j}$ through a point of an orbit $\M_\pk$ is the same for all points of~$\M_\pk$.

  \item[(ii)]
  The total number  $\Lb_{\lambda,\pk}$ of lines from an orbit union $\OO_\lambda$ through a point of an orbit $\M_\pk$ is the same for all points of~$\M_\pk$. We have
   \begin{align}\label{eq4_class_Lambda}
&  \Lb_{\lambda,\pk}=\sum_{j=1}^{L_{\lambda\mathrm{\Sigma}}} \Lb_{\lambda_j,\pk}.
  \end{align}
  \item[(iii)] The number $\Pb_{\pk,\lambda_j}$ of points from an orbit $\M_\pk$ on a line of an orbit $\OO_{\lambda_j}$ is the same for all lines of $\OO_{\lambda_j}$.

  \item[(iv)] The average
number $\Pb_{\pk,\lambda}$ of points from an orbit $\M_\pk$ on a line of a union $\OO_\lambda$ over all lines of $\OO_\lambda$ satisfies the following relations:
\begin{align}\label{eq4:Pi_aver}
&\Lb_{\lambda,\pk}\cdot\# \M_\pk=\Pb_{\pk,\lambda}\cdot\#\OO_\lambda;\db\\
&\Pb_{\pk,\lambda}=\frac{1}{\#\OO_{\lambda}}\sum\limits_{j=1}^{L_{\lambda\mathrm{\Sigma}}}\left(\Pb_{\pk,\lambda_j}\cdot\#\OO_{\lambda_j}\right).\label{eq4:Pi_aver2}
\end{align}

 \item[(v)]
If $L_{\lambda\mathrm{\Sigma}}=1$, then $\OO_\lambda$ is an orbit and the
number  of points from $\M_\pk$ on a line of $\OO_\lambda$ is
the same for all the lines of $\OO_\lambda$.
In this case, $\Pb_{\pk,\lambda}$ is certainly an integer.
     If $\Pb_{\pk,\lambda}$ is not an integer then the class $\OO_\lambda$ contains more than one orbit, i.e. $L_{\lambda\mathrm{\Sigma}}\ge2$.
\end{description}
\end{lemma}

\begin{proof}
\begin{description}
  \item[(i)]
Consider points $\pk_1$ and $\pk_2$ of~$\M_\pk$. Denote by $\ell$ a line of $\OO_{\lambda_j}$. Let $S(\pk_1)$ and $S(\pk_2)$ be the subsets of $\OO_{\lambda_j}$ such that $S(\pk_1)=\{\ell\in\OO_{\lambda_j}|\pk_1\in\ell\}$, $S(\pk_2)=\{\ell\in\OO_{\lambda_j}|\pk_2\in\ell\}$. There exists $\varphi\in G_q$  such that  $\pk_2=\pk_1\varphi$. Clearly, $\varphi$ embeds $S(\pk_1)$ in $S(\pk_2)$, i.e. $S(\pk_1)\varphi\subseteq S(\pk_2)$ and $\#S(\pk_1)\le\#S(\pk_2)$. In the same way, $\varphi^{-1}$ embeds $S(\pk_2)$ in $S(\pk_1)$, i.e.  $\#S(\pk_2)\le\#S(\pk_1)$. Thus,  $\#S(\pk_2)=\#S(\pk_1)$.

  \item[(ii)] For a fixed $\lambda$, orbits $\OO_{\lambda_j}$ do not intersect each other.

  \item[(iii)] The assertion can be proved similarly to the case (i).

  \item[(iv)] The cardinality $C_1$ of the multiset consisting of the lines of $\OO_\lambda$ through all the points of $\M_\pk$ is equal to $\Lb_{\lambda,\pk}\cdot\# \M_\pk$. The cardinality $C_2$ of the multiset consisting of the points of $\M_\pk$ on all the lines of $\OO_\lambda$ is  $\Pb_{\pk,\lambda}\cdot\#\OO_\lambda$. Every $C_i$ is the number of ones in the incidence submatrix $\I_{\pk,\lambda}^{\Pb\Lb}$ of $\I^{\Pb\Lb}$.  Thus, $C_1=C_2$.

      The assertion \eqref{eq4:Pi_aver2} holds as $\OO_\lambda$ is \emph{partitioned} into $L_{\lambda\mathrm{\Sigma}}$ orbits $\OO_{\lambda_j}$.

  \item[(v)] The assertion follows from the case (iii). \qedhere
\end{description}
 \end{proof}

\begin{corollary}\label{cor4_=0}
  If  $\Pb_{\pk,\lambda}=0$ then $\Lb_{\lambda,\pk}=0$ and vice versa.
\end{corollary}

\begin{proof}
    The assertions follow from \eqref{eq4:Pi_aver}.
\end{proof}

\begin{theorem}\label{th4_linepoint}
Let the lines of $\PG(3,q)$ be partitioned under $G_q$ into $\#\Lk$ classes $\OO_\lambda$ where every class is a union of orbits of $\lambda$-lines, $\lambda\in\Lk$. Also, let $\PG(3,q)$ be partitioned under $G_q$ by $\#\Mk$ orbits $\M_\pk$ of $\pk$-points, $\pk\in\Mk$.
The following holds:
\begin{align}
 &\sum_{\pk\in\Mk} \Pb_{\pk,\lambda}=q+1,~\lambda\T{ is fixed};\label{eq4_points_in_line_sum}\db\\
&\sum_{\lambda\in\Lk}\Lb_{\lambda,\pk}=\beta_{2,q}=q^2+q+1,~\pk\T{ is fixed}.\label{eq4_lines_through point_sum}
\end{align}
 \end{theorem}

\begin{proof}
         Relations \eqref{eq4_points_in_line_sum} and \eqref{eq4_lines_through point_sum} hold as  $\PG(3,q)$ is \emph{partitioned}  under $G_q$ by unions of line orbits and  by orbits of points. In total, in $\PG(3,q)$, there are $q+1$ points on every line and $\beta_{2,q}$ lines through every point.
\end{proof}

\begin{corollary}\label{cor4_obtainPbLb}
  The following holds:
  \begin{align}\label{eq4_obtainPb}
  &\Pb_{\pk,\lambda}=\frac{\Lb_{\lambda,\pk}\cdot\#\M_\pk}{\#\OO_\lambda},~\Pb_{\pk,\lambda_j}=\frac{\Lb_{\lambda_j,\pk}\cdot\#\M_\pk}{\#\OO_{\lambda_j}};\db\\
  &\Lb_{\lambda,\pk}=\frac{\Pb_{\pk,\lambda}\cdot\#\OO_\lambda}{\#\M_\pk},~
  \Lb_{\lambda_j,\pk}=\frac{\Pb_{\pk,\lambda_j}\cdot\#\OO_{\lambda_j}}{\#\M_\pk};\db\label{eq4_obtainLb}\\
  &\Pb_{\pk^*,\lambda}= q+1-\sum_{\pk\in\Mk\setminus\{\pk^*\}} \Pb_{\pk,\lambda},~\lambda\T{ is fixed},~\pk^*\in\Mk;\label{eq4_obtainPb2}\db\\
  &\Pb_{\pk^*,\lambda_j}= q+1-\sum_{\pk\in\Mk\setminus\{\pk^*\}} \Pb_{\pk,\lambda_j},~\lambda_j\T{ is fixed},~\pk^*\in\Mk;\label{eq4_obtainPb3}\db\\
&\Lb_{\lambda^*,\pk}= q^2+q+1-\sum_{\lambda\in\Lk\setminus\{\lambda^*\}}\Lb_{\lambda,\pk},~\pk\T{ is fixed},~\lambda^*\in\Lk.\label{eq4_obtainLb2}
  \end{align}
\end{corollary}

\begin{proof}
The assertions directly follow from \eqref{eq4:Pi_aver}, \eqref{eq4_points_in_line_sum}, \eqref{eq4_lines_through point_sum}.
\end{proof}

\begin{remark} \label{observation4:EA}
 Let $q\equiv0\pmod3$. By Section \ref{sec_prelimin},  $\mathrm{\Gamma}$-planes form a pencil with the $\Ar$-line as the axis. Only lines lying in a
  $\mathrm{\Gamma}$-plane can intersect the axis.

  By definition, an $\EA$-line necessary intersects the $\Ar$-line; therefore an $\EA$-line always lies in a
  $\mathrm{\Gamma}$-plane and intersects all the other lines belonging to this plane including the only tangent. Also, in \cite[Tables 1, 2, Theorem 3.3(iv), Corollary 7.2]{DMP_PlLineInc} it is proved that we have exactly one $\mathrm{\Gamma}$-plane through every $\EA$-line and, in every $\mathrm{\Gamma}$-plane, there are $q^2-1$ $\EA$-lines such that $q^2-q$ from them belong to the orbit $\OO_{\EA_1}$ while the remaining $q-1$ ones are equally divided into the orbits $\OO_{\EA_2}$, $\OO_{\EA_3}$.
 In total, in every $\mathrm{\Gamma}$-plane, we have $q^2-1$ intersections of $\EA$-lines and the $\Ar$-line.

 In addition, by definition, every $\mathrm{\Gamma}$-plane contains a tangent and $q$ $\UG$-lines intersecting the $\Ar$-line in distinct points. Thus, in every $\mathrm{\Gamma}$-plane, through a $(q+1)_\mathrm{\Gamma}$-point (i.e. a point of the $\Ar$-line) we have a unisecant, the $\Ar$-line, and $q-1$ $\EA$-lines.
\end{remark}

\begin{remark}\label{observation4:IC}
  By \cite[Table 1, Theorem 3.3(vi)]{DMP_PlLineInc}, all $q+1$ planes through an imaginary chord are $\overline{1_\C}$-planes forming a pencil. The $\binom{q}{2}(q+1)$-orbit of all  $\overline{1_\C}$-planes can be partitioned into $\binom{q}{2}$ pencils of planes having an imaginary chord as the axis. Only lines lying in a   $\overline{1_\C}$-plane can intersect an $\IC$-line. If, in average, there are $\mathrm{\Pi}_{\overline{1_\C},\lambda}$ $\overline{1_\C}$-planes through a $\lambda$-line ($\lambda\ne\IC$) then every $\lambda$-line intersects, in average, $\mathrm{\Pi}_{\overline{1_\C},\lambda}$  $\IC$-lines and contains, in average, $\mathrm{\Pi}_{\overline{1_\C},\lambda}$  $\IC$-points. So,
  \begin{align}\label{eq4:IC_P}
   \Pb_{\IC,\lambda}=\mathrm{\Pi}_{\overline{1_\C},\lambda},~\lambda\in\Lk.
  \end{align}
   whence, by \eqref{eq4_obtainLb}, Theorem~\ref{th2_Hirs}(ii)(a)(c), and \cite[equation (4.9)]{DMP_PlLineInc}, we have
  \begin{align}\label{eq4:IC_L}
   \Lb_{\lambda,\IC}=\frac{\Pb_{\IC,\lambda}\cdot\#\OO_\lambda}{\#\M_\IC}=
     \frac{\mathrm{\Pi}_{\overline{1_\C},\lambda}\cdot\#\OO_\lambda}{\#\N_{\overline{1_\C}}}=\mathrm{\Lambda}_{\lambda,\overline{1_\C}},~\lambda\in\Lk.
  \end{align}

  Similarly, for the $j$-th orbit $\OO_{\lambda_j}$ we have
  \begin{align}\label{eq4:IC_lambdaj}
    \Pb_{\IC,\lambda_j}=\mathrm{\Pi}_{\overline{1_\C},\lambda_j},~\Lb_{\lambda_j,\IC}=\mathrm{\Lambda}_{\lambda_j,\overline{1_\C}},
    ~j=1,\ldots,L_{\lambda\mathrm{\Sigma}},~\lambda\in\Lk.
      \end{align}
  The values of $\mathrm{\Pi}_{\overline{1_\C},\lambda}$, $\mathrm{\Pi}_{\overline{1_\C},\lambda_j}$, $\mathrm{\Lambda}_{\lambda,\overline{1_\C}}$, and $\mathrm{\Lambda}_{\lambda_j,\overline{1_\C}}$ can be taken from \cite[Tables~1,~2]{DMP_PlLineInc}.
   \end{remark}

\begin{theorem}\label{th4:NiU=Mi}
  Let $q\not\equiv0\pmod3$. Let $\pi\in\Pk$, $\pk\in\Mk^{(\ne0)}$, and $\{\lambda_a,\lambda_b\}\subset\Lk^{(\ne0)}$ be such that $\M_\pk^{(\ne0)}\A=\N_\pi$ and $\OO_{\lambda_a}=\OO_{\lambda_b}\A$. Then
  \begin{align}\label{eq4::NiU=Mi}
  \M_\pk^{(\ne0)}=\N_\pi\A,~~ \OO_{\lambda_a}\A=\OO_{\lambda_b}.
     \end{align}
\end{theorem}

\begin{proof}
By definition, see \cite[Sections 2.1.5, 5.3]{Hirs_PGFF}, a polarity is involutory, i.e. $\A^2 =\mathfrak{J}$, where $\mathfrak{J}$ is the identity. Therefore,
$\A^{-1}=\A$.
\end{proof}

\begin{corollary}\label{cor4:NiU=Mi_all}
  Let $q\not\equiv0\pmod3$. The following holds:
 \begin{align}\label{eq4:NiU=Mi}
&  \M_j^{(\ne0)}=\N_j\A, ~\#\M_j^{(\ne0)}=\#\N_j,~j=1,\ldots,5;~\M_\C^{(\ne0)}=\N_\mathrm{\Gamma}\A, \db\\
&\M_\Tr^{(\ne0)}=\N_{2_\C}\A,~\M_{3_\mathrm{\Gamma}}^{(\ne0)}=\N_{3_\C}\A,\,\M_{1_\mathrm{\Gamma}}^{(\ne0)}=\N_{\overline{1_\C}}\A,~
\M_{0_\mathrm{\Gamma}}^{(\ne0)}=\N_{0_\C}\A;\dbn\\
  &\OO'_i\A=\OO_i,~\#\OO'_i=\#\OO_i,~i=1,\ldots,6,~q\not\equiv0\pmod3;\db\label{eq4:O'iU=Oi}\\
  &\OO_\RA\A=\OO_\RC,~\OO_\IA\A=\OO_\IC,~\OO_\EG\A=\OO_\UnG; ~\OO_\lambda\A=\OO_\lambda,~\lambda\in\{\Tr,\UG,\EnG\}.\nt
 \end{align}
\end{corollary}

\begin{proof}
We use \eqref{eq2:MiU=Ni}, Table \ref{tab1}, and Theorems \ref{th2_Hirs}(ii)(iii), \ref{th2:orbLine}, \ref{th4:NiU=Mi}.
\end{proof}

\section{The numbers of $\lambda$-lines through $\pk$-points and of $\pk$-points on $\lambda$-lines, $q\not\equiv0\pmod3$}\label{sec:results_q_ne0}
Remind that we consider $q\ge5$; also $q\equiv\xi\pmod3$.

\textbf{Notation 3}
~In addition to Notations 1 and 2 we denote the following:
\begin{align*}
&\pi(\pk)\in\Pk &&\T{the plane type such that }\M_\pk^{(\ne0)}\A=\N_{\pi(\pk)},~\pk\in\Mk^{(\ne0)},~\xi\ne0;\db\\
&\lambda(\widetilde{\lambda})\in\Lk^{(\ne0)} &&\T{the line type such that }\OO_{\lambda(\widetilde{\lambda})}=\OO_{\widetilde{\lambda}}\A,~\widetilde{\lambda}\in\Lk^{(\ne0)},~\xi\ne0;\db\\
&\lambda_j(\widetilde{\lambda}_j)&&\T{the line type of the $j$-th orbit of the class $\OO_{\lambda(\widetilde{\lambda})}$ correspon-}\db\\
&&&\T{ding to the $j$-th orbit of the class $\OO_{\widetilde{\lambda}}$ so that }\OO_{\lambda_j(\widetilde{\lambda}_j)}=\OO_{\widetilde{\lambda}_j}\A.\nt
\end{align*}

\begin{theorem}\label{th5:pi(pk)lambda(widehatlambda)}
  Let $q\not\equiv0\pmod3$. The following holds:
  \begin{align}\label{eq5_pi(pk)}
  &\pi(\C)=\mathrm{\Gamma},~\pi(\Tr)=2_\C,~\pi(3_\mathrm{\Gamma})=3_\C,~\pi(1_\mathrm{\Gamma})=\overline{1_\C},~\pi(0_\mathrm{\Gamma})=0_\C;\db\\
  &\lambda(\RC)=\RA,~\lambda(\RA )=\RC,~\lambda(\Tr)=\Tr,~\lambda(\IC)=\IA,~\lambda(\IA)=\IC,~\dbn\\
  &\lambda(\UG)=\UG,~\lambda(\UnG)=\EG,~\lambda(\EG)=\UnG,~\lambda(\EnG)=\EnG;\dbn\\
  &\lambda_j(\UG_j)=\UG_j,~\lambda_j(\UnG_j)=\EG_j,~\lambda_j(\EG_j)=\UnG_j,~ j=1,2. \nt
  \end{align}
\end{theorem}

\begin{proof}
  The assertions directly follow from \eqref{eq2:MiU=Ni}, \eqref{eq2:O'=OU}, \eqref{eq4::NiU=Mi}--\eqref{eq4:O'iU=Oi}, Theorems \ref{th2_Hirs}(iii),  \ref{th2:orbLine}, \ref{th4:NiU=Mi}, and Corollary \ref{cor4:NiU=Mi_all}.
\end{proof}
\begin{theorem}\label{th5:Pi-->Pb}
  Let $q\not\equiv0\pmod3$. Let $\pk\in\Mk^{(\ne0)}$, $\widetilde{\lambda}\in\Lk^{(\ne0)}$.  Then
  \begin{align*}
 & \Pb_{\pk,\widetilde{\lambda}}^{(\xi)}=\mathrm{\Pi}_{\pi(\pk),\lambda(\widetilde{\lambda})}^{(\xi)},~\Lb_{\widetilde{\lambda},\pk}^{(\xi)}=
  \mathrm{\Lambda}_{\lambda(\widetilde{\lambda}),\pi(\pk)}^{(\xi)};\,
  \Pb_{\pk,\widetilde{\lambda}_j}^{(\xi)}=\mathrm{\Pi}_{\pi(\pk),\lambda_j(\widetilde{\lambda}_j)}^{(\xi)},\,\Lb_{\widetilde{\lambda}_j,\pk}^{(\xi)}=
  \mathrm{\Lambda}_{\lambda_j(\widetilde{\lambda}_j),\pi(\pk)}^{(\xi)}.
     \end{align*}
\end{theorem}

\begin{proof}
We have $\M_\pk^{(\ne0)}\A=\N_{\pi(\pk)}$,~$\OO_{\lambda(\widetilde{\lambda})}=\OO_{\widetilde{\lambda}}\A$. By Theorem \ref{th4:NiU=Mi}, $\M_\pk^{(\ne0)}=\N_{\pi(\pk)}\A$, $\OO_{\lambda(\widetilde{\lambda})}\A=\OO_{\widetilde{\lambda}}$. The incidences between $\pi(\pk)$-planes and $\lambda(\widetilde{\lambda})$-lines are saved for $\pk$-points and $\widetilde{\lambda}$-lines. The same holds for orbits $\OO_{\lambda_j}$.
\end{proof}

\begin{corollary}\label{cor5:for tables}
  Let $q\not\equiv0\pmod3$. Let $\pk\in\Mk^{(\ne0)}$. Let $\pi(\pk)\in\Pk$ be as in \eqref{eq5_pi(pk)}. For $\xi=1,-1$, the following holds:
  \begin{align*}
&\Pb_{\pk,\RC}^{(\xi)}=\mathrm{\Pi}_{\pi(\pk),\RA}^{(\xi)},~\Lb_{\RC,\pk}^{(\xi)}=\mathrm{\Lambda}_{\RA,\pi(\pk)}^{(\xi)},~
\Pb_{\pk,\RA}^{(\ne0)}=\mathrm{\Pi}_{\pi(\pk),\RC},~\Lb_{\RA,\pk}^{(\ne0)}=\mathrm{\Lambda}_{\RC,\pi(\pk)};\dbn\\
  &\Pb_{\pk,\IC}^{(\xi)}=\mathrm{\Pi}_{\pi(\pk),\IA}^{(\xi)},~\Lb_{\IC,\pk}^{(\xi)}=\mathrm{\Lambda}_{\IA,\pi(\pk)}^{(\xi)},
~\Pb_{\pk,\IA}^{(\ne0)}=\mathrm{\Pi}_{\pi(\pk),\IC},~\Lb_{\IA,\pk}^{(\ne0)}=\mathrm{\Lambda}_{\IC,\pi(\pk)};\dbn\\
  &\Pb_{\pk,\UnG}^{(\xi)}=\mathrm{\Pi}_{\pi(\pk),\EG}^{(\xi)},\,\Lb_{\UnG,\pk}^{(\xi)}=\mathrm{\Lambda}_{\EG,\pi(\pk)}^{(\xi)},
~\Pb_{\pk,\EG}^{(\ne0)}=\mathrm{\Pi}_{\pi(\pk),\UnG},\,\Lb_{\EG,\pk}^{(\ne0)}=\mathrm{\Lambda}_{\UnG,\pi(\pk)};\dbn\\
&\Pb_{\pk,\lambda}^{(\xi)}=\mathrm{\Pi}_{\pi(\pk),\lambda}^{(\xi)},~\Lb_{\lambda,\pk}^{(\xi)}=
\mathrm{\Lambda}_{\lambda,\pi(\pk)}^{(\xi)},~\lambda\in\{\Tr,\UG,\EnG\}.\dbn\\
&\Pb_{\pk,\UG_j}^{(\ne0)}=\mathrm{\Pi}_{\pi(\pk),\UG_j},~\Lb_{\UG_j,\pk}^{(\ne0)}=\mathrm{\Lambda}_{\UG_j,\pi(\pk)},~ \Pb_{\pk,\UnG_j}^{(\xi)}=\mathrm{\Pi}_{\pi(\pk),\EG_j}^{(\xi)},\db\\
&\Lb_{\UnG_j,\pk}^{(\xi)}=\mathrm{\Lambda}_{\EG_j,\pi(\pk)}^{(\xi)},~
\Pb_{\pk,\EG_j}^{(\ne0)}=\mathrm{\Pi}_{\pi(\pk),\UnG_j},~\Lb_{\EG_j,\pk}^{(\ne0)}=\mathrm{\Lambda}_{\UnG_j,\pi(\pk)},~ j=1,2.
  \end{align*}
\end{corollary}

\begin{proof}
  We use Theorems \ref{th5:pi(pk)lambda(widehatlambda)} and \ref{th5:Pi-->Pb}.
\end{proof}

Now we are able to form Tables \ref{tab2} and \ref{tab3} using Corollary \ref{cor5:for tables}, and the values of $\mathrm{\Pi}_{\pi,\lambda}^{(\xi)}$ and
  $\mathrm{\Lambda}_{\lambda,\pi}^{(\xi)}$ from \cite[Tables 1, 2]{DMP_PlLineInc}.

\section{The numbers of $\lambda$-lines through $\pk$-points and of $\pk$-points on $\lambda$-lines, $q\equiv0\pmod3$}\label{sec:results_q=0}
In this section, we consider $q\equiv0\pmod3$.
The values of $\#\M_\pk$, $\#\OO_\lambda$, $\#\OO_{\lambda_j}$, needed for  \eqref{eq4_obtainPb}, \eqref{eq4_obtainLb}, are taken from \eqref{eq2_point_orbits_gen}--\eqref{eq2_=0_orbit_point} and Table~\ref{tab1}. When we use \eqref{eq4_obtainPb2}--\eqref{eq4_obtainLb2}, the values $\Pb_{\pk,\lambda}$, $\Pb_{\pk,\lambda_j}$, and $\Lb_{\lambda,\pk}$, obtained above, are summed~up.

Note that if some of the relations \eqref{eq4_class_Lambda}--\eqref{eq4:IC_lambdaj} are not used directly in proofs then they can be used to check the results.
\begin{theorem}\label{th6:RC}
  For $\RC$-lines the following holds:
  \begin{align*}
 &\Pb_{\C,\RC}^{(0)}=2,~\Pb_{\RC,\RC}^{(0)}=q-1,~\Lb_{\RC,\C}^{(0)}=q,~\Lb_{\RC,\RC}^{(0)}=1,~\db\\
 & \Pb_{\pk,\RC}^{(0)}=\Lb_{\RC,\pk}^{(0)}=0,~
 \pk\in\{(q+1)_\mathrm{\Gamma},\TO,\IC\}.
  \end{align*}
\end{theorem}

\begin{proof}
 By definition, a real chord contains exactly two $\C$-points; the other $q-1$ points of the chord are $\RC$-points. So, $\Pb_{\C,\RC}^{(0)}=2,~\Pb_{\RC,\RC}^{(0)}=q-1$. Also, by definition, there are $q$ $\RC$-lines through every $\C$-point, i.e. $\Lb_{\RC,\C}^{(0)}=q$.

 By Theorem \ref{th2_Hirs}(iv), two chords do not intersect each other off $\C$; therefore, $\Lb_{\RC,\RC}^{(0)}=1$, $\Pb_{\pk,\RC}^{(0)}=\Lb_{\RC,\pk}^{(0)}=0,~
 \pk\in\{\TO,\IC\}$.

Finally, by Remark \ref{observation4:EA}, only lines lying in a
  $\mathrm{\Gamma}$-plane can intersect the $\Ar$-line. As an $\RC$-line contains two $\C$-points, it cannot lie in a $\mathrm{\Gamma}$-plane, whence $\Pb_{(q+1)_\mathrm{\Gamma},\RC}^{(0)}=\Lb_{\RC,(q+1)_\mathrm{\Gamma}}^{(0)}=0$.
 \end{proof}

\begin{theorem}\label{th6:T}
  For $\Tr$-lines the following holds:
  \begin{align*}
 &\Pb_{\C,\Tr}^{(0)}=\Pb_{(q+1)_\mathrm{\Gamma},\Tr}^{(0)}=\Lb_{\Tr,\C}^{(0)}=\Lb_{\Tr,(q+1)_\mathrm{\Gamma}}^{(0)}=\Lb_{\Tr,\TO}^{(0)}=1,~\Pb_{\TO,\Tr}^{(0)}=q-1,\db\\
 &\Pb_{\pk,\Tr}^{(0)}= \Lb_{\Tr,\pk}^{(0)}=0,~\pk\in\{\RC,\IC\}.
  \end{align*}
\end{theorem}

\begin{proof}
By definition, a tangent contains exactly one $\C$-point and there is one tangent through every $\C$-point. Thus, $\Pb_{\C,\Tr}^{(0)}=\Lb_{\Tr,\C}^{(0)}=1$. Also, a tangent lies in a $\mathrm{\Gamma}$-plane and, hence, intersects the $\Ar$-line. This implies
$\Pb_{(q+1)_\mathrm{\Gamma},\Tr}^{(0)}=1$ whence, by~\eqref{eq4_obtainLb}, $\Lb_{\Tr,(q+1)_\mathrm{\Gamma}}^{(0)}=1$. The remaining  $q-1$ points of the tangent are $\TO$-points, i.e. $\Pb_{\TO,\Tr}^{(0)}=q-1$.

 By Theorem \ref{th2_Hirs}(iv), two chords do not intersect each other off $\C$; therefore, $\Lb_{\Tr,\TO}^{(0)}=1$,
 $\Pb_{\RC,\Tr}^{(0)}=\Pb_{\IC,\Tr}^{(0)}= \Lb_{\Tr,\RC}^{(0)}=\Lb_{\Tr,\IC}^{(0)}=0$.
\end{proof}

\begin{theorem}\label{th6:IC}
 For $\IC$-lines the following holds:
  \begin{align*}
 &\Pb_{\IC,\IC}^{(0)}=q+1,~\Lb_{\IC,\IC}^{(0)}=1,~\Pb_{\pk,\IC}^{(0)}=\Lb_{\IC,\pk}^{(0)}=0,~\pk\in\{\C,(q+1)_\mathrm{\Gamma},\TO,\RC\}.
  \end{align*}
\end{theorem}

\begin{proof}
By definition, all points of an $\IC$-line are $\IC$-points; this implies $\Pb_{\IC,\IC}^{(0)}=q+1$,  $\Pb_{\pk,\IC}^{(0)}=\Lb_{\IC,\pk}^{(0)}=0,~\pk\in\{\C,(q+1)_\mathrm{\Gamma},\TO,\RC\}$.
By Theorem \ref{th2_Hirs}(iv), two $\IC$-lines do not intersect each other; therefore, $\Lb_{\IC,\IC}^{(0)}=1$.
\end{proof}

\begin{theorem}\label{th6:UG}
 For $\UG$-lines the following holds:
  \begin{align*}
 &\Pb_{\C,\UG}^{(0)}=\Pb_{(q+1)_\mathrm{\Gamma},\UG}^{(0)}=1,~\Lb_{\UG,\C}^{(0)}=\Lb_{\UG,(q+1)_\mathrm{\Gamma}}^{(0)}=q,~\Pb_{\TO,\UG}^{(0)}=\Lb_{\UG,\TO}^{(0)}=0,\db\\
 &\Lb_{\UG,\RC}^{(0)}=\Lb_{\UG,\IC}^{(0)}=1,~\Pb_{\RC,\UG}^{(0)}=\Pb_{\IC,\UG}^{(0)}=\frac{1}{2}(q-1).
  \end{align*}
\end{theorem}

\begin{proof}
  By definition, a $\UG$-line contains exactly one $\C$-point and there are $q$ $\UG$-lines thro\-ugh every $\C$-point, i.e. $\Pb_{\C,\UG}^{(0)}=1$, $\Lb_{\UG,\C}^{(0)}=q$. A $\UG$-line lies in a $\mathrm{\Gamma}$-plane and, hence, intersects the $\Ar$-line. This implies $\Pb_{(q+1)_\mathrm{\Gamma},\UG}^{(0)}=1$ whence, by \eqref{eq4_obtainLb}, $\Lb_{\UG,(q+1)_\mathrm{\Gamma}}^{(0)}=q$.

 By definition, $\UG$-lines and $\Tr$-lines lie  in $\mathrm{\Gamma}$-planes. If a $\UG$-line and a $\Tr$-line belong to the same $\mathrm{\Gamma}$-plane, then their common point  is a $\C$-point. Otherwise they are skew. So, a $\UG$-line cannot intersect a $\Tr$-line off $\C$.  As all $\TO$-points are off $\C$, we have $\Pb_{\TO,\UG}^{(0)}=\Lb_{\UG,\TO}^{(0)}=0$.

 By \cite[Table 2, Theorem 5.13(ii)]{BDMP-TwCub}, there is exactly one $\mathrm{\Gamma}$-plane, say $\pi_P$, through an $\RC$-point $P$. Let $Q\in\pi_P$ be the contact point of $\C$ and $\pi_P$. By Theorem \ref{th2_Hirs}(iv), the line $\overline{PQ}$ cannot be either a real chord or a tangent, hence $\overline{PQ}$  is a $\UG$-line. Thus, $\Lb_{\UG,\RC}^{(0)}=1$. Similarly, it can be shown that $\Lb_{\UG,\IC}^{(0)}=1$. Now, by \eqref{eq4_obtainPb}, we obtain $\Pb_{\RC,\UG}^{(0)}=\Pb_{\IC,\UG}^{(0)}=(q-1)/2$.
\end{proof}

\begin{theorem}\label{th6:A}
 For the $\Ar$-line the following holds:
  \begin{align*}
 &\Pb_{(q+1)_\mathrm{\Gamma},\Ar}^{(0)}=q+1,~\Lb_{\Ar,(q+1)_\mathrm{\Gamma}}^{(0)}=1,~\Pb_{\pk,\Ar}^{(0)}=\Lb_{\Ar,\pk}^{(0)}=0,~\pk\in\{\C,\TO,\RC,\IC\}.
  \end{align*}
\end{theorem}

\begin{proof}
  By definition, all points of the $\Ar$-line are $(q+1)_\mathrm{\Gamma}$-points and there is one $\Ar$-line through every $(q+1)_\mathrm{\Gamma}$-point; this implies $\Pb_{(q+1)_\mathrm{\Gamma},\Ar}^{(0)}=q+1$, $\Lb_{\Ar,(q+1)_\mathrm{\Gamma}}^{(0)}=1$, and $\Pb_{\pk,\Ar}^{(0)}=\Lb_{\Ar,\pk}^{(0)}=0$, $\pk\ne(q+1)_\mathrm{\Gamma}$.
\end{proof}

\begin{theorem}\label{th6:UnGa}
 The following holds:
  \begin{align*}
 &\Pb_{\C,\UnG}^{(0)}=\Pb_{\C,\UnG_v}^{(0)}=1,~\Lb_{\UnG,\C}^{(0)}=q^2-q,~\Lb_{\UnG_v,\C}^{(0)}=\frac{1}{2}(q^2-q),~v=1,2;\db\\
 &\Pb_{\C,\lambda}^{(0)}= \Lb_{\lambda,\C}^{(0)}=0,~\lambda\in\{\EnG,\EA,\EA_j\},~j=1,2,3.
  \end{align*}
\end{theorem}

\begin{proof}
By definition, a $\UnG$-line contains exactly one $\C$-point, i.e. $\Pb_{\C,\UnG}^{(0)}=\\
\Pb_{\C,\UnG_v}^{(0)}=1$, whence, by \eqref{eq4_obtainLb}, $\Lb_{\UnG,\C}^{(0)}=q^2-q$, $\Lb_{\UnG_v,\C}^{(0)}=(q^2-q)/2$. Also, by definition, $\Pb_{\C,\lambda}^{(0)}= 0$, $\lambda\in\{\EnG,\EA,\EA_j\}$,
whence, by Corollary \ref{cor4_=0}, $\Lb_{\lambda,\C}^{(0)}=0$.
\end{proof}

\begin{theorem}\label{th6:(q+1)_Gamma}
  A $\UnG$- and an $\EnG$-line cannot intersect the $\Ar$-line whereas an $\EA$-line necessary intersects it. The following holds:
  \begin{align*}
  &\Pb_{(q+1)_\mathrm{\Gamma},\lambda}^{(0)}=\Lb_{\lambda,(q+1)_\mathrm{\Gamma}}^{(0)}=0,~\lambda\in\{\UnG,\UnG_v,\EnG\},~v=1,2;\db\\
  &  \Pb_{(q+1)_\mathrm{\Gamma},\lambda}^{(0)}=1,~\lambda\in\{\EA,\EA_j\},~j=1,2,3;\db\\
  &\Lb_{\EA,(q+1)_\mathrm{\Gamma}}^{(0)}=q^2-1, ~\Lb_{\EA_1,(q+1)_\mathrm{\Gamma}}^{(0)}=q^2-q,~\Lb_{\EA_j,(q+1)_\mathrm{\Gamma}}^{(0)}=\frac{1}{2}(q-1),~j=2,3;
 \end{align*}
 where $\Pb_{(q+1)_\mathrm{\Gamma},\EA}^{(0)}=1$ is the \emph{exact} number of $(q+1)_\mathrm{\Gamma}$-points on an $\EA$-line.
 \end{theorem}

\begin{proof}
  By definition, $\UnG$- and $\EnG$-lines do not lie in any $\mathrm{\Gamma}$-plane; it implies, due to Remark \ref{observation4:EA}, $ \Pb_{(q+1)_\mathrm{\Gamma},\lambda}^{(0)}=\Lb_{\lambda,(q+1)_\mathrm{\Gamma}}^{(0)}=0,~\lambda\in\{\UnG,\UnG_v,\EnG\}$. Also, by definition, an $\EA$-line necessary intersects the $\Ar$-line that gives $\Pb_{(q+1)_\mathrm{\Gamma},\lambda}^{(0)}=1,~\lambda\in\{\EA,\EA_j\}$, as the exact value. Now, by \eqref{eq4_obtainLb}, we obtain $\Lb_{\EA,(q+1)_\mathrm{\Gamma}}^{(0)}$ and $\Lb_{\EA_j,(q+1)_\mathrm{\Gamma}}^{(0)}$.
  \end{proof}

\begin{theorem}\label{th6:UnG}
   For $\UnG$-lines the following holds:
  \begin{align*}
  &\Lb_{\UnG,\TO}^{(0)}=\Lb_{\UnG,\IC}^{(0)}=q,~\Pb_{\TO,\UnG}^{(0)}=1,~\Pb_{\IC,\UnG}^{(0)}=\frac{1}{2}q,\db\\
&\Lb_{\UnG,\RC}^{(0)}=q-2,\,\Pb_{\RC,\UnG}^{(0)}=
  \frac{1}{2}(q-2).
 \end{align*}
\end{theorem}

\begin{proof}
Let $\TT$  be a tangent to $\C$ at a point $P$. Let $B\in\TT$ be a $\TO$-point. Let $\ell$ be a line through $B$ and one of the $q$ points of $\C\setminus \{P\}$. By Theorem~\ref{th2_Hirs}(iv), $\ell$ can be neither a real chord nor a tangent, hence it is a non-tangent unisecant. By\cite[Table 2, Theorem 5.13(ii)]{BDMP-TwCub}, there is exactly one $\mathrm{\Gamma}$-plane, say $\pi_B$, through the $\TO$-point $B$. Obviously, $\TT\in\pi_B$ and $P$ is the contact point of  $\C$ and $\pi_B$.
Thus, $\ell$ does not lie in a $\mathrm{\Gamma}$-plane, i.e. $\ell$ is a $\UnG$-line and we have $\#\C\setminus\{P\}$ $\UnG$-lines through every $\TO$-point. So, $\Lb_{\UnG,\TO}^{(0)}=q$, whence, by~\eqref{eq4_obtainPb}, $\Pb_{\TO,\UnG}^{(0)}=1$.

Let $\overline{PQ}$ be a real chord through $\C$-points $P$ and $Q$. Let $B\in\overline{PQ}$ be an $\RC$-point. By\cite[Table 2, Theorem 5.13(ii)]{BDMP-TwCub}, there is exactly one $\mathrm{\Gamma}$-plane, say $\pi_B$, through the $\RC$-point~$B$. Let $R\in\pi_B$ be the contact point of $\C$ and $\pi_B$. Let $\ell$ be a line through $P$ and one of the $q-2$ points of $\C\setminus\{P,Q,R\}$. By Theorem~\ref{th2_Hirs}(iv), $\ell$ can be neither a real chord nor a tangent; also, $\ell\notin\pi_B$. Thus, $\ell$ is a $\UnG$-line and we have $\#\C\setminus\{P,Q,R\}$ $\UnG$-lines through every $\RC$-point. So, $\Lb_{\UnG,\RC}^{(0)}=q-2$, whence, by~\eqref{eq4_obtainPb}, $\Pb_{\RC,\UnG}^{(0)}=(q-2)/2$.

Finally, let $\mathcal{IC}$ be an imaginary chord and $B\in\mathcal{IC}$ be an $\IC$-point. By \cite[Table~2]{BDMP-TwCub}, there is exactly one $\mathrm{\Gamma}$-plane, say $\pi_B$, through $B$. Let $R\in\pi_B$ be the contact point of $\C$ and $\pi_B$. Similarly to above, all the $q$ lines through $B$ and a point of $\C\setminus\{R\}$ are $\UnG$-lines; this gives $\Lb_{\UnG,\IC}^{(0)}=q$, whence, by~\eqref{eq4_obtainPb}, $\Pb_{\IC,\UnG}^{(0)}=q/2$.
\end{proof}

\begin{theorem}\label{th6:TO}
   For $\TO$-points the following holds:
  \begin{align*}
  &\Pb_{\TO,\EA}^{(0)}=\frac{q}{q+1},~\Lb_{\EA,\TO}^{(0)}=\Lb_{\EA_1,\TO}^{(0)}=q,~\Lb_{\EnG,\TO}^{(0)}=q^2-q,\\
  &\Pb_{\TO,\EA_1}^{(0)}=  \Pb_{\TO,\EnG}^{(0)}=1,~\Pb_{\TO,\EA_j}^{(0)}= \Lb_{\EA_j,\TO}^{(0)}=0,~j=2,3.
 \end{align*}
\end{theorem}

\begin{proof}
By Remark \ref{observation4:EA}, in every $\mathrm{\Gamma}$-plane, $q-1$ $\EA$-lines intersect the only tangent at its common point with the $\Ar$-line while the remaining $q^2-q$ ones intersect the tangent in $\TO$-points. Thus, in total, there are $(q^2-q)\cdot\#\N_\mathrm{\Gamma}=(q^2-q)(q+1)$ $\TO$-points on all $(q+1)(q^2-1)$ $\EA$-lines. The average number is $\Pb_{\TO,\EA}^{(0)}=q/(q+1)$, whence, by \eqref{eq4_obtainLb}, $\Lb_{\EA,\TO}^{(0)}=q$.

An $\EA$-line  intersects exactly one tangent either  in its common point with the $\Ar$-line or in a $\TO$-point. Therefore, $\Pb_{\TO,\EA_j}^{(0)}\in\{0,1\}$. If for $j=2,3$, we put $\Pb_{\TO,\EA_j}^{(0)}=1$ then, by \eqref{eq4_obtainLb}, we obtain $\Lb_{\EA_j,\TO}^{(0)}=1/2$ that is not an integer, contradiction.
So, $\Pb_{\TO,\EA_j}^{(0)}= \Lb_{\EA_j,\TO}^{(0)}=0$, $j=2,3$, whence, by \eqref{eq4:Pi_aver2}, $\Pb_{\TO,\EA_1}^{(0)}= 1$ and, by \eqref{eq4_obtainLb}, $\Lb_{\EA_1,\TO}^{(0)}=q$.

Finally, by \eqref{eq4_obtainLb2}, $\Lb_{\EnG,\TO}^{(0)}=q^2-q$, whence, by \eqref{eq4_obtainPb}, $\Pb_{\TO,\EnG}^{(0)}=1$.
\end{proof}

\begin{remark}\label{observation6}
 By Remark \ref{observation4:EA} and Theorem \ref{th6:TO}, it can be seen that in every $\mathrm{\Gamma}$-plane, the $q-1$ $\EA$-lines from the orbits $\OO_{\EA_2}$ and $\OO_{\EA_3}$
intersect the only tangent of this plane at its common point with the $\Ar$-line (it is not a $\TO$-point). At the same time, the $q^2-q$ $\EA$-lines from the orbit $\OO_{\EA_1}$ intersect the tangent in $\TO$-points.
\end{remark}

\begin{theorem}\label{th6:IC}
   For $\IC$-points the following holds:
  \begin{align*}
  &   \Pb_{\IC,\EnG}^{(0)}=\frac{q^2-q-1}{2(q-1)},~\Lb_{\EnG,\IC}^{(0)}=q^2-q-1;~
  \Pb_{\IC,\EA}^{(0)}=\frac{q^2}{2(q+1)},~\Lb_{\EA,\IC}^{(0)}=q;\db\\
  &\Pb_{\IC,\EA_1}^{(0)}=\frac{1}{2}(q-1),~\Lb_{\EA_1,\IC}^{(0)}=q-1,~
  \Pb_{\IC,\EA_2}^{(0)}=\Lb_{\EA_2,\IC}^{(0)}=0,\db\\
  &\Pb_{\IC,\EA_3}^{(0)}=q,~\Lb_{\EA_3,\IC}^{(0)}=1.
 \end{align*}
\end{theorem}

\begin{proof}
  The assertions follow from Remark \ref{observation4:IC} with \eqref{eq4:IC_P}--\eqref{eq4:IC_lambdaj}  and \cite[Tables~1,~2]{DMP_PlLineInc}.
\end{proof}

\begin{theorem}\label{th6:RC_RC}
   For $\RC$-points the following holds:
  \begin{align*}
  &  \Pb_{\RC,\EnG}^{(0)}=\frac{q^2-q+1}{2(q-1)},~\Lb_{\EnG,\RC}^{(0)}=q^2-q+1;\db\\
  &\Pb_{\RC,\EA}^{(0)}=\frac{q^2}{2(q+1)},~\Lb_{\EA,\RC}^{(0)}=q ;~\Pb_{\RC,\EA_1}^{(0)}=\frac{1}{2}(q-1),~\Lb_{\EA_1,\RC}^{(0)}=q-1,\db\\ &\Pb_{\RC,\EA_2}^{(0)}=q,~\Lb_{\EA_2,\RC}^{(0)}=1,~\Pb_{\RC,\EA_3}^{(0)}=\Lb_{\EA_3,\IC}^{(0)}=0.
 \end{align*}
\end{theorem}

\begin{proof}
  The values of $\Pb_{\RC,\lambda}^{(0)}$, $\Pb_{\RC,\lambda_j}^{(0)}$ are obtained by \eqref{eq4_obtainPb2}, \eqref{eq4_obtainPb3}. Then we obtain
$\Lb_{\lambda,\RC}^{(0)}$, $\Lb_{\lambda_j,\RC}^{(0)}$ by \eqref{eq4_obtainLb}.
\end{proof}

\begin{theorem}
  For $\UnG_v$-lines, $v=1,2$, the following holds:
  \begin{align*}
    &\Pb_{\TO,\UnG_1}^{(0)}=\Lb_{\UnG_1,\TO}^{(0)}=0,~\Pb_{\TO,\UnG_2}^{(0)}=2,~\Lb_{\UnG_2,\TO}^{(0)}=q,\db\\
     &\Pb_{\IC,\UnG_1}^{(0)}=\Lb_{\UnG_1,\IC}^{(0)}=\frac{1}{2}(q+1),~\Pb_{\IC,\UnG_2}^{(0)}=\Lb_{\UnG_2,\IC}^{(0)}=\frac{1}{2}(q-1),\db\\
&\Pb_{\RC,\UnG_1}^{(0)}=\Lb_{\UnG_1,\RC}^{(0)}=\frac{1}{2}(q-1),~
     \Pb_{\RC,\UnG_2}^{(0)}=\Lb_{\UnG_2,\RC}^{(0)}=\frac{1}{2}(q-3).
  \end{align*}
\end{theorem}

\begin{proof}
By Theorem \ref{th6:TO}, $\Pb_{\TO,\EnG}^{(0)}=1$ whence $\Pb_{\TO,\UnG_1}^{(0)}+\Pb_{\TO,\UnG_2}^{(0)}=2$, by \eqref{eq4:Pi_aver2}. If for $v=1,2$, we put $\Pb_{\TO,\UnG_v}^{(0)}=1$, then, by \eqref{eq4_obtainLb}, we obtain $\Lb_{\UnG_v.\TO}^{(0)}=q/2$ that is not an integer, contradiction. So, $\Pb_{\TO,\UnG_v}^{(0)}\in\{0,2\}$.

By \cite[Table 1]{DMP_PlLineInc}, through a $\Tr$-line we have one $\mathrm{\Gamma}$-plane and $q$ $2_\C$-planes; also, every $2_\C$-plane contains one $\Tr$-line. Therefore, through a $\Tr$-line and a $\UnG$-line meeting the $\Tr$-line in a $\TO$-point we have a $2_\C$-plane. By \cite[Table~2]{DMP_PlLineInc}, there are one and three $2_\C$-planes through a $\UnG_1$- and a $\UnG_2$-line, respectively. Therefore, a $\UnG_1$- and a $\UnG_2$-line intersect one and three $\Tr$-lines, respectively. This means that $\Pb_{\TO,\UnG_1}^{(0)}=0$, $\Pb_{\TO,\UnG_2}^{(0)}=2$, and also every $\UnG_v$-line intersects one  $\Tr$-line at a $\C$-point. Now, by~\eqref{eq4_obtainLb}, we obtain $\Lb_{\UnG_1,\TO}^{(0)}=0,~\Lb_{\UnG_2,\TO}^{(0)}=q$.

By \cite[Table 2]{DMP_PlLineInc}, $\mathrm{\Pi}_{\overline{1_\C},\UnG_1}=\mathrm{\Lambda}_{\UnG_1,\overline{1_\C}}=(q+1)/2, \mathrm{\Pi}_{\overline{1_\C},\UnG_2}=\mathrm{\Lambda}_{\UnG_2,\overline{1_\C}}=(q-1)/2$ whence, by \eqref{eq4:IC_lambdaj}, $\Pb_{\IC,\UnG_1}^{(0)}=\Lb_{\UnG_1,\IC}^{(0)}=(q+1)/2$, $\Pb_{\IC,\UnG_2}^{(0)}=\Lb_{\UnG_2,\IC}^{(0)}=(q-1)/2$.

Finally, $\RC$-points lie only on $\RC$-lines. By \cite[Table 1]{DMP_PlLineInc}, through an $\RC$-line there are two $2_\C$-planes and $q-1$ $3_\C$-planes.
A $\UnG$-line lying in a $2_\C$-plane intersects the $\RC$-line of this plane at a $\C$-point. A $\UnG$-line lying in a $3_\C$-plane intersects two $\RC$-lines of this plane at a $\C$-point and one $\RC$-line at an $\RC$-point. Therefore, $\Pb_{\TO,\UnG_v}^{(0)}=\mathrm{\Pi}_{3_\C,\UnG_v}^{(0)}$. By \cite[Table~2]{DMP_PlLineInc}, $\mathrm{\Pi}_{3_\C,\UnG_1}^{(0)}=(q-1)/2,~\mathrm{\Pi}_{3_\C,\UnG_2}^{(0)}=(q-3)/2$ whence together with \eqref{eq4_obtainLb} we obtain the remaining assertions.
\end{proof}

Now we form Tables \ref{tab4} and \ref{tab5} using the results of this section.

\section{Some general results }\label{sec:gen res}
\begin{theorem}
 Let $\pk\in\Mk^{(\xi)}$. Let a class $\OO_\lambda$ consist of a single orbit.
Then the submatrix $\I^{\Pb\Lb}_{\pk,\lambda}$ of  $\I^{\Pb\Lb}$ is
  a $(v_r,b_k)$ configuration of Definition \emph{\ref{def2_config}} with $v=\#\M_\pk$, $b=\#\OO_\lambda$, $r=\Lb_{\lambda,\pk}^{(\xi)}$, $k=\Pb_{\pk,\lambda}^{(\xi)}$. Also, up to rearrangement of rows and columns, the submatrices $\I^{\Pb\Lb}_{\pk,\lambda}$ with $\Lb_{\lambda,\pk}^{(\xi)}=1$ can be viewed as a concatenation of $\Pb_{\pk,\lambda}^{(\xi)}$ identity matrices of order $\#\OO_\lambda$. The same holds for the submatrices $\I^{\Pb\Lb}_{\pk,\lambda_j}$.
\end{theorem}

\begin{proof}
As the  class $\OO_\lambda$ is an orbit, $\I^{\Pb\Lb}_{\pk,\lambda}$ contains $\Pb_{\pk,\lambda}$ (resp. $\Lb_{\lambda,\pk}$) ones in every row (resp. column), see Lemma~\ref{lemma4_line&point}. In $\PG(3,q)$, two lines are either skew or intersect at a point. Therefore, two points of  $\I^{\Pb\Lb}_{\pk,\lambda}$ are connected by at most one line. If $\Lb_{\lambda,\pk}^{(\xi)}=1$, $\I^{\Pb\Lb}_{\pk,\lambda}$ contains $\Pb_{\pk,\lambda}^{(\xi)}$ (resp.\ 1) ones in every row (resp. column).
\end{proof}

\begin{theorem}
Let $(\lambda,\pk)\in\{(\UG,\C), (\UnG,\C)\}$ if $q\not\equiv0\pmod3$, and
 $(\lambda,\pk)\in\{(\UnG,\C), (\EA,(q+1)_\mathrm{\Gamma})\}$ if $q\equiv0\pmod3$. Then,  independently of the number of orbits in the class $\OO_\lambda$, we have exactly one $\pk$-point on every $\lambda$-line. Up to rearrangement of rows and columns, the submatrices $\I^{\Pb\Lb}_{\pk,\lambda}$ can be viewed as a vertical concatenation of\/ $\Lb_{\lambda,\pk}^{(\xi)}$ identity matrices of order $\#\M_\pk$.
\end{theorem}

\begin{proof}
 By Tables \ref{tab2}--\ref{tab5}, in the considered submatrices $\I^{\Pb\Lb}_{\pk,\lambda}$ there is exactly one unit in every row.
\end{proof}

\begin{theorem}\label{th7:one line}
  Let $q\equiv1\pmod3$. Let $V^{(1)}=\{\OO_1=\OO_\RC,\OO_2=\OO_\Tr,\OO'_3=\OO_\IA\}$. Then, cf. Theorem $\ref{th2_Hirs}$(iv), no two lines of $V^{(1)}$ meet off $\C$. Every point off $\C$ lies on exactly one line of~$V^{(1)}$.
\end{theorem}

\begin{proof}
  The assertions follow from Table \ref{tab2}. Also, they can be proved using Theorem~\ref{th2_Hirs}(ii)(iv) with \eqref{eq2_=1_orbit_point} and \cite[Theorem 1]{LunarPolv}.
\end{proof}

\begin{theorem}
  Let $q\equiv0\pmod3$. Let $\W^{(0)}=\{\OO_2=\OO_\Tr,\OO_4=\OO_\UG\}$. Let $\mathbb{M}=\C\cup\Ar$-line be the union of the twisted cubic and the $\Ar$-line.  Then
  no two lines of $\W^{(0)}$ meet off $\mathbb{M}$.
 Every point off $\mathbb{M}$ lies on exactly one line of~$\W^{(0)}$, cf. Theorems $\ref{th2_Hirs}$(iv) and \emph{\ref{th7:one line}}.
\end{theorem}

\begin{proof}
  The assertions follow from Table \ref{tab4}. Note also that for $q\equiv0\pmod3$, $\mathrm{\Gamma}$-planes form a pencil with the axis $\Ar$-line and, in turn, in each plane $\pi_\T{osc}(t)$, $q$ $\UG$-lines and the tangent form the pencil of lines through the point $P(t)$.
\end{proof}

\section*{Acknowledgments}
The research of S. Marcugini, and F. Pambianco was supported in part by the Italian
National Group for Algebraic and Geometric Structures and their Applications (GNSAGA -
INDAM) (Contract No. U-UFMBAZ-2019-000160, 11.02.2019) and by University of Perugia
(Project No. 98751: Strutture Geometriche, Combinatoria e loro Applicazioni, Base Research
Fund 2017-2019).


\begin{thebibliography}{99}
\bibitem{BallLavrauw} Ball, S., Lavrauw, M.:  Arcs in finite projective spaces, EMS Surv. Math. Sci. \textbf{6}(1/2), 133--172 (2019)

\bibitem{BargZem}
Barg, A., Zemor, G.: Distances properties of expander codes, IEEE Trans. Inf.
  Theory \textbf{52}(1), 78--90 (2006)

\bibitem{BDMP-TwCub} Bartoli, D., Davydov, A.A., Marcugini, S., Pambianco, F.:
On planes through points off the twisted cubic in PG(3,q) and multiple covering codes, Finite Fields Appl. \textbf{67}, 101710, 25 pages (2020)

\bibitem{BlokPelSzo} Blokhuis, A., Pellikaan, R., Sz\"{o}nyi, T.: The extended coset leader weight enumerator
of a twisted cubic code, arXiv:2103.16904 [math.CO] (2021)

\bibitem{BonPolvTwCub} Bonoli, G., Polverino, O.: The twisted cubic in $\PG(3, q)$ and translation
spreads in $H(q)$, Discrete Math. \textbf{296}, 129--142 (2005)

 \bibitem{Magma}  Bosma, W.,  Cannon, J., Playoust, C.: The Magma Algebra System. I. The User Language, J. Symbolic Comput. \textbf{24}, 235--265 (1997)

\bibitem{BrHirsTwCub}  Bruen, A.A., Hirschfeld, J.W.P.: Applications of line geometry over finite fields I: The twisted cubic, Geom. Dedicata \textbf{6},495--509  (1977)

\bibitem{CLPolvT_Spr} Cardinali, I., Lunardon, G., Polverino, O., Trombetti, R.: Spreads in $H(q)$ and 1-systems of $Q(6,q)$,
European J. Combin. \textbf{23}, 367--376 (2002)

\bibitem{CasseGlynn82}  Casse, L.R.A., Glynn, D.G.: The solution to Beniamino Segre's problem $I_{r,q}$, $r = 3$, $q = 2^h$, Geom.
Dedicata \textbf{13}, 157--163 (1982)

\bibitem{CasseGlynn84}Casse, L.R.A., Glynn, D.G.: On the uniqueness of $(q + 1)_{4}$-arcs of $\PG(4,q)$, $ q= 2^h$, $h\ge 3$,
Discrete Math. \textbf{48}(2-3), 173--186 (1984)

\bibitem{CosHirsStTwCub}   Cossidente, A., Hirschfeld, J.W.P., Storme, L.: Applications of line geometry, III:
The quadric Veronesean and the chords of a twisted cubic, Austral. J. Combin. \textbf{16}, 99--111 (1997)

\bibitem{DFGMP_SymConf} Davydov, A.A., Faina, G., Giulietti, M., Marcugini, S., Pambianco, F.:
On constructions and parameters of symmetric configurations $v_k$, Des. Codes Cryptogr. \textbf{80}(1), 125--147 (2016)

\bibitem{DGMP_BipGraph}
Davydov, A.A., Giulietti, M., Marcugini, S., Pambianco, F.: Some combinatorial
  aspects of constructing bipartite-graph codes, Graphs Combin. \textbf{29}(2), 187--212 (2013)


\bibitem{DMP_RSCoset} Davydov, A.A., Marcugini, S., Pambianco, F.: On cosets weight distributions of the doubly-extended Reed-Solomon codes of codimension 4, IEEE Trans. Inform. Theory, to appear,
 doi: 10.1109/TIT.2021.3089129.

\bibitem{DMP_PlLineInc} Davydov, A.A., Marcugini, S., Pambianco, F.: Twisted cubic and plane-line incidence matrix in $\PG(3,q)$, 	arXiv:2103.11248 [math.CO] (2021)

\bibitem{DMP_OrbLine} Davydov, A.A., Marcugini, S., Pambianco, F.: Twisted cubic and orbits of lines in $\PG(3,q)$,
arXiv:2103.12655 [math.CO] (2021)

\bibitem{GiulVincTwCub}  Giulietti, M.,  Vincenti, R.: Three-level secret sharing schemes from the twisted cubic,
Discrete Math. \textbf{310}, 3236--3240 (2010)

 \bibitem{GulLav} G\"{u}nay, G., Lavrauw, M.: On pencils of cubics on the projective line over finite fields of characteristic $> 3$,
 arXiv:2104.04756 [math.CO][math.AG] (2021)

\bibitem{GroppConfig} Gropp, H.: Configurations. In: Colbourn, C.J., Dinitz,  J.H. (eds.) Handbook of Combinatorial Designs, 2nd  edition,  pp. 353--355. Chapman and Hall/CRC, New York (2006)

\bibitem{Hirs_PGFF} Hirschfeld, J.W.P.: Projective Geometries over Finite Fields, 2nd edition, Oxford Univ.
Press, Oxford (1999)

\bibitem{Hirs_PG3q} Hirschfeld, J.W.P.: Finite Projective Spaces of Three Dimensions, Oxford Univ. Press, Oxford (1985)

\bibitem{HirsStor-2001} Hirschfeld, J.W.P.,  Storme, L.:   The
    packing problem in statistics, coding theory and finite projective spaces: Update 2001, in: A. Blokhuis, J.W.P. Hirschfeld, D. Jungnickel, J.A. Thas (Eds.), Finite
    Geometries (Proc. 4th Isle of Thorns Conf., July 16-21, 2000),
    Dev. Math., vol. 3, Dordrecht: Kluwer, 2001, pp. 201--246

\bibitem{HirsThas-2015}Hirschfeld, J.W.P.,  Thas, J.A.:
 Open  problems in finite projective spaces,   Finite Fields
    Appl. \textbf{32}, 44--81 (2015)

    \bibitem{HohJust}
H{\o}holdt, T., Justesen, J.: Graph codes with {Reed--Solomon} component codes,
  in: Proc. Int. Symp. Inf. Theory 2006, ISIT 2006, IEEE, Seattle, WA, USA,
  2006, pp. 2022--2026

\bibitem{Lidl_Nied} Lidl, R.,  Niederreiter, H.: Introduction to Finite Fields and their Applications, 2nd edition, Cambridge University Press, 1994.

\bibitem{LunarPolv} Lunardon, G., Polverino, O.:  On the Twisted Cubic of $\PG(3, q)$, J. Algebr. Combin. \textbf{18}, 255--262 (2003)

\bibitem{HandbCombDes2v_k_lamb}
Mathon, R., Rosa, A.: $2$-$(v,k,\lambda)$ Designs of Small Order. In: Colbourn, C.J., Dinitz,  J.H. (eds.) Handbook of Combinatorial Designs, 2nd
  edition, pp. 25--58. Chapman and Hall/CRC, New York (2006)

\bibitem{ZanZuan2010}  Zannetti, M.,  Zuanni, F.: Note on three-character $(q + 1)$-sets in $\PG(3, q)$,  Austral. J. Combin. \textbf{47}, 37--40 (2010)
\end{thebibliography}
\end{document}